\documentclass[12pt,a4paper,twoside]{article}

\usepackage{fancyhdr,graphicx,subfigure,psfrag}
\usepackage{amsmath,amsfonts,amsthm}
\usepackage[abs]{overpic}
\usepackage{epsfig}

\newcommand{\Lie}{\mathcal{L}}
\newcommand{\tr}{\mathrm{tr}}

\newcommand{\p}{\partial}

\newtheorem{theorem}{Theorem}
\newtheorem{lemma}{Lemma}

\newtheorem{remark}{Remark}
\newtheorem{proposition}{Proposition}

\numberwithin{equation}{section}
\begin{document}
\title{Universality in Complex Wishart ensembles: The 1 cut case}
\author{M. Y. Mo}
\date{}
\maketitle
\begin{abstract} We considered $N\times N$ Wishart ensembles
in the class $W_\mathbb{C}\left(\Sigma_N,M\right)$ (complex Wishart
matrices with $M$ degrees of freedom and covariance matrix
$\Sigma_N$) such that $N_0$ eigenvalues of $\Sigma_N$ is 1 and
$N_1=N-N_0$ of them are $a$. We studied the limit as $M$, $N$, $N_0$
and $N_1$ all go to infinity such that $\frac{N}{M}\rightarrow c$,
$\frac{N_1}{N}\rightarrow\beta$ and $0<c,\beta<1$. In this case, the
limiting eigenvalue density can either be supported on 1 or 2
disjoint intervals in $\mathbb{R}_+$. In the previous paper
\cite{Mo}, we studied the universality in the case when the limiting
eigenvalue density is supported on 2 intervals and in this paper, we
continue the analysis and study the case when the support consists
of a single interval. By using Riemann-Hilbert analysis, we have
shown that under proper rescaling of the eigenvalues, the limiting
correlation kernel is given by the sine kernel and the Airy kernel
in the bulk and the edge of the spectrum respectively. As a
consequence, the behavior of the largest eigenvalue in this model is
described by the Tracy-Widom distribution.
\end{abstract}
\section{Introduction}
Let $X$ be an $M\times N$ (assuming $M\geq N$) matrix with i.i.d.
complex Gaussian entries whose real and imaginary parts have
variance $\frac{1}{2}$ and zero mean. Let $\Sigma_N$ be an $N\times
N$ positive definite Hermitian matrix with eigenvalues
$a_1,\ldots,a_N$ (not necessarily distinct). In this paper, we will
consider the case where $\Sigma_N$ has only 2 distinct eigenvalues,
1 and $a$ such that $N_1$ of its eigenvalues are $a$ and $N-N_1$ of
them are $1$. We will assume that $\frac{N}{M}\rightarrow c$ and
$\frac{N_1}{N}\rightarrow\beta$ as $N,M\rightarrow\infty$ and that
$0<c,\beta<1$. To be precise, we will assume the following
\begin{equation}\label{eq:lim}
cM-N=\tau_1=O(1),\quad N\beta-N_1=\tau_2=O(1),\quad
M,N,N_1\rightarrow\infty.
\end{equation}
Let $\Sigma_N^{\frac{1}{2}}$ be any Hermitian square root of
$\Sigma_N$. Then the columns of the matrix $X\Sigma_N^{\frac{1}{2}}$
are random vectors with variances $\frac{1}{2}\sqrt{a_j}$. Let the
matrix $B_N$ be the following
\begin{equation}\label{eq:BN}
B_N=\frac{1}{M}\Sigma_N^{\frac{1}{2}}X^{\dag}X\Sigma_N^{\frac{1}{2}},
\end{equation}
Then $B_N$ is the sample covariance matrix of the columns of
$X\Sigma_N^{\frac{1}{2}}$, while $\Sigma_N$ is the covariance
matrix. In particular, $B_N$ is a complex Wishart matrix in the
class $W_{\mathbb{C}}\left(\Sigma_N,M\right)$.

The sample covariance matrix is a fundamental tool in the studies of
multivariate statistics and its distribution is already known to
Wishart at around 1928 (See e.g. \cite{Muir})
\begin{equation}\label{eq:wishart}
\mathcal{P}(B_N)=\frac{1}{C}e^{-M\tr(\Sigma^{-1}B_N)}\left(\det
S\right)^{M-N},\quad M\geq N,
\end{equation}
for some normalization constant $C$.

Let $y_1,\ldots,y_N>0$ be the eigenvalues of the sample covariance
matrix $B_N$. Then in the case where $N_1$ eigenvalues of $\Sigma_N$
is $a$ and $N-N_1$ is $1$, the joint probability density function
(j.p.d.f) for the eigenvalues of $B_N$ is given by
\begin{equation}\label{eq:jpdf1}
\mathcal{P}(y)=\frac{1}{Z_{M,N}}\prod_{i<j}(y_i-y_j)\prod_{j=1}^Ny_{j}^{M-N}\det\left[
y_k^{d^N_j-1}e^{-Ma_j^{-1}y_k}\right]_{1\leq j,k,\leq N},
\end{equation}
where $Z_{M,N}$ is a normalization constant and $d^N_j$, $a_j$ are
given by
\begin{equation*}
\begin{split}
d^N_j&=j,\quad a_j=1,\quad 1\leq j\leq N-N_1,\\
d^N_j&=j-N+N_1,\quad a_j=a,\quad N-N_1< j\leq N.
\end{split}
\end{equation*}
In this paper we will study the asymptotic limit of the Wishart
distribution with $\frac{N}{M}\rightarrow c$ and
$\frac{N_1}{N}\rightarrow\beta$ as $M$, $N\rightarrow\infty$ in such
a way that $0<\beta,c<1$. In this case, the empirical distribution
function (e.d.f) $F_N$ of the eigenvalues will converge weakly to a
nonrandom p.d.f. $F$, which will be supported on either 1 or 2
intervals in $\mathbb{R}_+$. By applying the results of \cite{CS} to
our case, we can extract properties of the measure $F$ from the
solutions of an algebraic equation (See Section \ref{se:Stie} for
details)
\begin{equation}\label{eq:curve20}
\begin{split}
za\xi^3&+(A_2z+B_2)\xi^2+(z+B_1)\xi+1=0,\\
A_2&=(1+a),\quad B_2=a(1-c),\\
B_1&=1-c(1-\beta)+a(1-c\beta).
\end{split}
\end{equation}
The results in \cite{CS} imply that the real zeros of the function
$\frac{d z(\xi)}{d\xi}$ determines the boundary points of the
support of $F$. Since the zeros of $\frac{d z(\xi)}{d\xi}$ coincide
with the zeros of the following quartic polynomial,
\begin{equation}\label{eq:quartic1}
\begin{split}
a^2(1-c)\xi^4
&+2(a^2(1-c\beta)+a(1-c(1-\beta))\xi^3\\
&+(1-c(1-\beta)+a^2(1-c\beta)+4a)\xi^2 +2(1+a)\xi+1=0,
\end{split}
\end{equation}
the real roots of (\ref{eq:quartic1}) are important in the
determination of $\mathrm{Supp}(F)$. In particular, we have the
following result. (See Theorem \ref{thm:cuts})
\begin{theorem}\label{thm:main1}
Let $\Delta$ be the discriminant of the quartic polynomial
(\ref{eq:quartic1}). If $\Delta<0$, then the support of $F$ consists
of a single intervals.
\end{theorem}
In our previous paper \cite{Mo}, we have shown that
$\mathrm{Supp}(F)$ consists of 2 disjoint intervals if and only if
$\Delta>0$ and have obtained the asymptotic eigenvalue statistics in
that case. In this paper, we consider the case when $\Delta<0$ and
together with \cite{Mo}, we have proven the universality in this
class of complex Wishart ensemble for all $\Delta\neq 0$. When
$\Delta=0$, a phase transition takes place and the support of the
eigenvalues splits into 2 disjoint intervals.

When $\Delta<0$, we also have the following expression for the
density function of $F$ (See Theorem \ref{thm:density}).
\begin{theorem}\label{thm:side1}
Let $\Delta$ be the discriminant of the quartic polynomial
(\ref{eq:quartic1}). Suppose $\Delta<0$ and let $\gamma_1<\gamma_2$
be the 2 real roots to (\ref{eq:quartic1}) and $\gamma_3$,
$\gamma_4$ be the 2 complex roots. Let $\lambda_k$, $k=1,\ldots,4$
be the following
\begin{equation*}
\lambda_k=-\frac{1}{\gamma_k}+c\frac{1-\beta}{1+\gamma_k}+c\frac{a\beta}{1+a\gamma_k},\quad
k=1,\ldots,4.
\end{equation*}
Then all the $\lambda_k$ are distinct and the p.d.f $F$ is supported
on $[\lambda_1,\lambda_2]$ with the following density $d
F(z)=\rho(z)dz$
\begin{equation}\label{eq:rho}
\begin{split}
\rho(z)=\frac{3}{2\pi}\left|\left(\frac{r(z)+\sqrt{-\frac{1}{27a^4z^4}D_3(z)}}{2}\right)^{\frac{1}{3}}-\left(\frac{r(z)-\sqrt{-\frac{1}{27a^4z^4}D_3(z)}}{2}\right)^{\frac{1}{3}}\right|,
\end{split}
\end{equation}
where $D_3(z)$ and $r(z)$ are given by
\begin{equation*}
\begin{split}
D_3(z)&=(1-a)^2\prod_{j=1}^4(z-\lambda_j),\\
r(z)&=\frac{1}{27}\Bigg(-\frac{2B_2^3}{a^3}z^{-3}+\left(\frac{9B_1B_2}{a^2}-\frac{6A_2B_2^2}{a^3}\right)z^{-2}+\left(\frac{9B_2}{a^2}+\frac{9B_1A_2}{a^2}-\frac{27}{a}-\frac{6A_2^2B_2}{a^3}\right)z^{-1}\\
&+\left(\frac{9A_2}{a^2}-\frac{2A_2^3}{a^3}\right)\Bigg).
\end{split}
\end{equation*}
The constants $A_1$, $B_1$ and $B_2$ in the above equation are
defined by
\begin{equation*}
A_2=(1+a),\quad B_2=a(1-c),\quad B_1=1-c(1-\beta)+a(1-c\beta).
\end{equation*}
The cube root in (\ref{eq:density}) is chosen such that
$\sqrt[3]{A}\in\mathbb{R}$ for $A\in\mathbb{R}$ and the square root
is chosen such that $\sqrt{A}>0$ for $A>0$.
\end{theorem}
\begin{remark}\label{re:sqrt}
Since $D_3(z)$ can be written as
\begin{equation}
\frac{D_3(z)}{a^4z^4}=-27\left(r(z)\right)^2-4\left(p(z)\right)^3,
\end{equation}
for some polynomial $p(z)$ in $z^{-1}$. We see that if $r(z)$
vanishes at any of the $\lambda_k$, then $D_3(z)$ will have at least
a double root at these points, hence $r(\lambda_k)\neq0$. From this
and (\ref{eq:rho}), we see that the density $\rho(z)$ vanishes like
a square root at the points $\lambda_k$, $k=1,2$.
\begin{equation}\label{eq:rhok}
\rho(z)=\frac{\rho_k}{\pi}|z-\lambda_k|^{\frac{1}{2}}+O\left((z-\lambda_k)\right),\quad
z\rightarrow\lambda_k.
\end{equation}
\end{remark}
An open problem in the studies of Wishart ensembles is the
universality and the distribution of the largest eigenvalue.
Although the Wishart distribution is known for a long time, results
in the universality and the largest eigenvalue distribution were
only obtained recently and only for spiked models \cite{J} whose
covariance matrices are finite perturbations of the identity matrix
\cite{BaikD}, \cite{Baik95}, \cite{Baikspike}, \cite{DF},
\cite{El03}, \cite{Fo93}, \cite{J}, \cite{Jo}, \cite{W1}, \cite{W2}.
The result in this paper is one of the few results obtained for
models with covariance matrices that is not a finite perturbation of
the identity matrix. (See also \cite{El}, in which the largest
eigenvalue distribution was also derived for ensembles whose
covariance matrix is not a finite perturbation of the identity
matrix. However, in \cite{El}, the parameters $c=\frac{N}{M}$,
$\beta=\frac{N_1}{N}$ and the eigenvalues of the covariance matrix
have to satisfy a condition which is not true in our case.)

In this paper, we use an important result by Baik, Ben-Arous and
P\'ech\'e \cite{Baik95} which shows that the correlation functions
of the eigenvalues can be expressed in terms of a Fredholm
determinant with kernel $K_{M,N}(x,y)$. In \cite{BK1} and \cite{DF},
the authors have expressed this kernel in terms of multiple
orthogonal polynomials (See Section \ref{se:MOP} for details) and
have shown that the $m$-point correlation function for the Wishart
distribution (\ref{eq:wishart}) is given by
\begin{equation}\label{eq:mpoint}
\mathcal{R}_{m}^{(M,N)}(y_1,\ldots,y_m)=\det\left(K_{M,N}(y_j,y_k)\right)_{1\leq
j,k\leq m}
\end{equation}
where $\mathcal{R}_{m}^{(M,N)}(y_1,\ldots,y_m)$ is the $m$-point
correlation function
\begin{equation}\label{eq:corre}
\mathcal{R}_m^{(M,N)}(y_1,\ldots,y_m)=\frac{N!}{(N-m)!}\int_{\mathbb{R}_+}\cdots\int_{\mathbb{R}_+}
\mathcal{P}(y)dy_{m+1}\ldots dy_N.
\end{equation}
By computing the asymptotics of the correlation kernel in
(\ref{eq:mpoint}) through the asymptotics of multiple Laguerre
polynomials, we have proven the universality of the correlation
function when $M$, $N$ and $N_1\rightarrow\infty$.
\begin{theorem}\label{thm:main2}
Suppose $\Delta$ in Theorem \ref{thm:main1} is negative. Let
$\rho(z)$ be the density function of $F$ in Theorem \ref{thm:side1}.
Then for any $x_0\in(\lambda_1,\lambda_2)$ and $m\in\mathbb{N}$, we
have
\begin{equation}\label{eq:bulk}
\begin{split}
\lim_{N,M\rightarrow\infty}&\left(\frac{1}{M\rho(x_0)}\right)^m
\mathcal{R}_m^{(M,N)}\left(x_0+\frac{u_1}{M\rho(x_0)},\ldots,x_0+\frac{u_m}{M\rho(x_0)}\right)\\
&=\det\left(\frac{\sin
\pi(u_i-u_j)}{\pi(u_i-u_j)}\right)_{i,j=1}^{m}.
\end{split}
\end{equation}
uniformly for any $(u_1,\ldots,u_m)$ in compact subsets of
$\mathbb{R}^m$.

On the other hand, let $x_0=\lambda_k$, $k=1,2$, then for any
$m\in\mathbb{N}$, we have
\begin{equation}\label{eq:edge}
\begin{split}
\lim_{N,M\rightarrow\infty}&\left(\frac{1}{\left(M\rho_k\right)^{\frac{2}{3}}}\right)^m
\mathcal{R}_m^{(M,N)}\left(\lambda_k+(-1)^k\frac{u_1}{\left(M\rho_k\right)^{\frac{2}{3}}},\ldots,
\lambda_k+(-1)^k\frac{u_m}{\left(M\rho_k\right)^{\frac{2}{3}}}\right)\\
&=\det\left(\frac{\mathrm{Ai}(u_i)\mathrm{Ai}^{\prime}(u_j)
-\mathrm{Ai}^{\prime}(u_i)\mathrm{Ai}(u_j)}{u_i-u_j}\right)_{i,j=1}^m,
\end{split}
\end{equation}
uniformly for any $(u_1,\ldots,u_m)$ in compact subsets of
$\mathbb{R}^m$, where $\mathrm{Ai}(z)$ is the Airy function and
$\rho_k$, $k=1,2$ are the constants in (\ref{eq:rhok}).
\end{theorem}
Recall that the Airy function is the unique solution to the
differential equation $v^{\prime\prime}=zv$ that has the following
asymptotic behavior as $z\rightarrow\infty$ in the sector
$-\pi+\epsilon\leq \arg(z)\leq \pi-\epsilon$, for any $\epsilon>0$.
\begin{equation}\label{eq:asymairy}
\mathrm{Ai}(z)=\frac{1}{2\sqrt{\pi}z^{\frac{1}{4}}}e^{-\frac{2}{3}z^{\frac{3}{2}}}\left(1+O(z^{-\frac{3}{2}})\right)
,\quad -\pi+\epsilon\leq \arg(z)\leq \pi-\epsilon,\quad
z\rightarrow\infty.
\end{equation}
where the branch cut of $z^\frac{3}{2}$ in the above is chosen to be
the negative real axis.

Since the limiting kernel takes the form of the Airy kernel
(\ref{eq:edge}), by a well-known result of Tracy and Widom
\cite{TW1}, we have the following
\begin{theorem}\label{thm:TW}
Let $y_1$ be the largest eigenvalue of $B_N$, then we have
\begin{equation}
\lim_{M,N\rightarrow\infty}
\mathbb{P}\left(\left(y_1-\lambda_2\right)\left(M\rho_2\right)^{\frac{2}{3}}<s\right)
=TW(s),
\end{equation}
where $TW(s)$ is the Tracy-Widom distribution
\begin{equation}\label{eq:TW}
TW(s)=\exp\left(-\int_{s}^{\infty}(t-s)q^2(t)dt\right),
\end{equation}
and $q(s)$ is the solution of Painlev\'e II equation
\begin{equation*}
q^{\prime\prime}(s)=sq(s)+2q^3(s),
\end{equation*}
with the following asymptotic behavior as $s\rightarrow\infty$.
\begin{equation*}
q(s)\sim-\mathrm{Ai}(s),\quad s\rightarrow +\infty.
\end{equation*}
\end{theorem}
The results obtained in this paper are obtained through the
Riemann-Hilbert analysis. As in \cite{BKext2} and \cite{Lysov}, a
Riemann surface of the form (\ref{eq:curve20}), together with the
zero set of a real function $h(x)$ related to this Riemann surface
are essential to the implementation of the Riemann-Hilbert analysis.
This function $h(x)$ is given as follows. If we express the
solutions $\xi$ in (\ref{eq:curve20}) as analytic functions of $z$,
then the three solutions to (\ref{eq:curve20}) behave as follows
when $z\rightarrow\infty$.
\begin{equation}\label{eq:xiinfty}
\begin{split}
\xi_1(z)&=-\frac{1}{z}+O(z^{-2}), \quad z\rightarrow\infty,\\
\xi_2(z)&=-1+\frac{c(1-\beta)}{z}+O(z^{-2}), \quad z\rightarrow\infty,\\
\xi_3(z)&=-\frac{1}{a}+\frac{c\beta}{z}+O(z^{-2}), \quad
z\rightarrow\infty.
\end{split}
\end{equation}
Then the function $h(x)$ is defined by
\begin{equation}\label{eq:hx}
h(x)=\mathrm{Re}\left(\int_{\lambda_3}^x\xi_2(z)-\xi_3(z)dz\right)
\end{equation}
In order to implement the Riemann-Hilbert analysis, we must
determine the sheet structure of the Riemann surface
(\ref{eq:curve20}) and the topology of the zero set $\frak{H}$ of
$h(x)$. Since our model depends on three parameters $c$, $a$ and
$\beta$, while the models in \cite{BKext2} and \cite{Lysov} depend
only on one parameter $a$, the determination of both the sheet
structure of the Riemann surface and the topology of $\frak{H}$ are
considerably more difficult in our case and a large part of this
paper is devoted to resolve these difficulties so that
Riemann-Hilbert analysis like those in \cite{BKext2} and
\cite{Lysov} can be applied.

To determine the sheet structure of (\ref{eq:curve20}), note that
although we assume the discriminant $\Delta$ in Theorem
\ref{thm:main1} is negative, which means that the Riemann surface
(\ref{eq:curve20}) has 2 real and 2 complex branch points, it is
unclear which Riemann sheet do these branch points belong to and
these different situations will result in different sheet structures
of the Riemann surface as indicated in Figure \ref{fig:struc}.
\begin{figure}
\centering \psfrag{x1}[][][1][0.0]{\small$\xi_1$}
\psfrag{x2}[][][1][0.0]{\small$\xi_2$}
\psfrag{x3}[][][1][0.0]{\small$\xi_3$}
\includegraphics[scale=0.5]{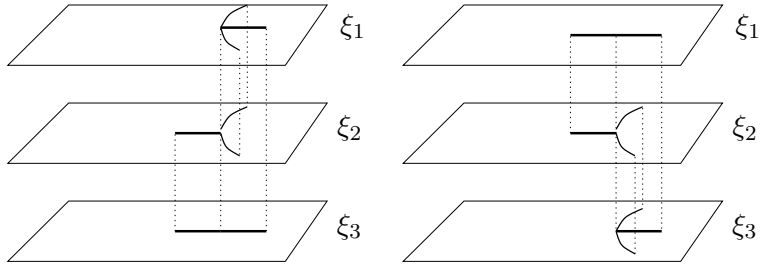}
\caption{Possible sheet structures of the Riemann surface. For the
implementation of the Riemann-Hilbert analysis, we need to show that
the Riemann surface has the sheet structure shown on the right hand
side.}\label{fig:struc}
\end{figure}
In order to determine the sheet structure of the Riemann surface, we
need to analyze the analyticity of the solutions $\xi_j$ in the
vicinity of the points $\lambda_k$ in Theorem \ref{thm:side1}. This
requires the type of analysis used in \cite{Mc}, which is very
difficult to carry out in our case. As in \cite{Mo}, we overcome
these difficulties by showing that the Stieltjes transform of the
limiting eigenvalue distribution satisfies (\ref{eq:curve20}). Then
by using properties of the Stieltjes transform obtained in
\cite{BS95}, \cite{BS98}, \cite{BS99}, \cite{CS}, \cite{S95}, we
were able to determine the sheet structure of the Riemann surface
(\ref{eq:curve20}).

The determination of the topology of the set $\frak{H}$ in this case
is also much more complicated and a thorough analysis of this set
making use of properties of harmonic functions is carried out in
Section \ref{se:geo}.

For theoretical reasons, the model studied in this paper is crucial
in understanding of the phase transition behavior that occurs in
Wishart ensembles. (See \cite{Baik95}). When the 2 intervals in the
support of $F$ closes up, a phase transition takes place and the
correlation kernel will demonstrate a different behavior at the
point where the support closes up. With the Riemann-Hilbert
technique used in this paper, such behavior can be studied
rigorously as in \cite{BKdou} and the eigenvalue correlation
function near the critical point will be given by the Pearcey
kernel. For practical reasons, many covariance matrices appearing in
fields of science are not finite perturbations of the identity
matrix. In fact, covariance matrices that have groups of distinct
eigenvalues are accepted models in various areas such as the
correlation of genes in microarray analysis and the correlation of
the returns of stocks in finance.

\subsection*{Acknowledgement}
The author acknowledges A. Kuijlaars for pointing out reference
\cite{Lysov} to me and EPSRC for the financial support provided by
the grant EP/D505534/1.

\section{Multiple Laguerre polynomials and the correlation
kernel}\label{se:MOP}

The main tool in our analysis involves the use of multiple
orthogonal polynomials and the Riemann-Hilbert problem associated
with them. In this section we shall recall the results in \cite{BK1}
and \cite{DF} and express the correlation kernel $K_{M,N}(x,y)$ in
(\ref{eq:mpoint}) in terms of multiple Laguerre polynomials. In
Section \ref{se:RHP}, we will apply Riemann-Hilbert analysis to
obtain the asymptotics of these multiple Laguerre polynomials and
use them to prove Theorem \ref{thm:main2}.

We shall not define the multiple Laguerre polynomials in the most
general setting, but only define the ones that are relevant to our
set up. Readers who are interested in the theory of multiple
orthogonal polynomials can consult the papers \cite{Ap1},
\cite{Ap2}, \cite{BK1}, \cite{vanGerKuij}. Let $L_{n_1,n_2}(x)$ be
the monic polynomial such that
\begin{equation}\label{eq:multiop}
\begin{split}
&\int_{0}^{\infty}L_{n_1,n_2}(x)x^{i+M-N}e^{-Mx}dx=0,\quad i=0,\ldots,n_1-1,\\
&\int_0^{\infty}L_{n_1,n_2}(x)x^{i+M-N}e^{-Ma^{-1}x}dx=0,\quad
i=0,\ldots, n_2-1.
\end{split}
\end{equation}
and let $Q_{n_1,n_2}(x)$ be a function of the form
\begin{equation}\label{eq:qform}
Q_{n_1,n_2}(x)=A^1_{n_1,n_2}(x)e^{-Mx}+A^a_{n_1,n_2}(x)e^{-Ma^{-1}x},
\end{equation}
where $A^1_{n_1,n_2}(x)$ and $A^a_{n_1,n_2}(x)$ are polynomials of
degrees $n_1-1$ and $n_2-1$ respectively, and that $Q_{n_1,n_2}(x)$
satisfies the following
\begin{equation}\label{eq:qorth}
\int_0^{\infty}x^iQ_{n_1,n_2}(x)x^{M-N}dx=\left\{
                                            \begin{array}{ll}
                                              0, & \hbox{$i=0,\ldots,n_1+n_2-2$;} \\
                                              1, & \hbox{$i=n_1+n_2-1$.}
                                            \end{array}
                                          \right.
\end{equation}
The polynomial $L_{n_1,n_2}(x)$ is called the multiple Laguerre
polynomial of type II and the polynomials $A^1_{n_1,n_2}(x)$ and
$A^a_{n_1,n_2}(x)$ are called multiple Laguerre polynomials of type
I (with respect to the weights $x^{M-N}e^{-Mx}$ and
$x^{M-N}e^{-Ma^{-1}x}$) \cite{Ap1}, \cite{Ap2}. These polynomials
exist and are unique. Moreover, they admit integral expressions
\cite{BK1}.

Let us define the constants $h^{(1)}_{n_1,n_2}$ and
$h^{(2)}_{n_1,n_2}$ to be
\begin{equation}\label{eq:normalize}
\begin{split}
h^{(1)}_{n_1,n_2}&=\int_{0}^{\infty}L_{n_1,n_2}(x)x^{n_1+M-N}e^{-Mx}dx,\\
h^{(2)}_{n_1,n_2}&=\int_{0}^{\infty}L_{n_1,n_2}(x)x^{n_2+M-N}e^{-Ma^{-1}x}dx.
\end{split}
\end{equation}

Then the following result in \cite{BK1} and \cite{DF} allows us to
express the correlation kernel in (\ref{eq:mpoint}) in terms of a
finite sum of the multiple Laguerre polynomials.
\begin{proposition}\label{pro:CD}
The correlation kernel in $K_{M,N}(x,y)$ (\ref{eq:mpoint}) can be
expressed in terms of multiple Laguerre polynomials as follows
\begin{equation}\label{eq:ker}
\begin{split}
\left(xy\right)^{\frac{N-M}{2}}(x-y)K_{M,N}(x,y)&=L_{N_0,N_1}(x)Q_{N_0,N_1}(x)\\
&-\frac{h^{(1)}_{N_0,N_1}}{h^{(1)}_{N_0-1,N_1}}L_{N_0-1,N_1}(x)Q_{N_0+1,N_1}(x)\\
&-\frac{h^{(2)}_{N_0,N_1}}{h^{(2)}_{N_0,N_1-1}}L_{N_0,N_1-1}(x)Q_{N_0,N_1+1}(x)
\end{split}
\end{equation}
where $N_0=N-N_1$.
\end{proposition}
This result allows us to compute the limiting kernel once we obtain
the asymptotics for the multiple Laguerre polynomials.

\section{Stieltjes transform of the eigenvalue
distribution}\label{se:Stie}
In order to study the asymptotics of
the correlation kernel, we would need to know the asymptotic
eigenvalue distribution of the Wishart ensemble (\ref{eq:wishart}).
Let $F_N(x)$ be the empirical distribution function (e.d.f) of the
eigenvalues of $B_N$ (\ref{eq:BN}). The asymptotic properties of
$F_N(x)$ can be found by studying its Stieltjes transform.

The Stieltjes transform of a probability distribution function
(p.d.f) $G(x)$ is defined by
\begin{equation}\label{eq:stie}
m_G(z)=\int_{-\infty}^{\infty}\frac{1}{\lambda-z}dG(x), \quad
z\in\mathbb{C}^+=\left\{z\in\mathbb{C}:\mathrm{Im}(z)>0\right\}.
\end{equation}
Given the Stieljes transform, the p.d.f can be found by the
inversion formula
\begin{equation}\label{eq:inver}
G([a,b])=\frac{1}{\pi}\lim_{\epsilon\rightarrow
0^+}\int_{a}^b\mathrm{Im} \left(m_G(s+i\epsilon)\right)ds.
\end{equation}
The properties of the Stieltjes transform of $F_N(x)$ has been
studied by Bai, Silverstein and Choi in the papers \cite{CS},
\cite{BS95}, \cite{BS98}, \cite{BS99}, \cite{S95}. We will now
summarize the results that we need from these papers.

First let us denote the e.d.f of the eigenvalues of $\Sigma_N$ by
$H_N$, that is, we have
\begin{equation*}
dH_N(x)=\frac{1}{N}\sum_{j=1}^N\delta_{a_j}.
\end{equation*}
Furthermore, we assume that as $N\rightarrow\infty$, the
distribution $H_N$ weakly converges to a distribution function $H$.
Then as $N\rightarrow\infty$, the e.d.f $F_N(x)$ converges weakly to
a nonrandom p.d.f $F$, and that the Stieltjes transform $m_F$ of
$F(x)$ satisfies the following equation \cite{CS}, \cite{S95}
\begin{equation}\label{eq:algeq}
m_F(z)=\int_{\mathbb{R}}\frac{1}{t(1-c-czm_F)-z}dH(t).
\end{equation}
Let us now consider the closely related matrix $\underline{B}_N$
\begin{equation}\label{eq:BNline}
\underline{B}_N=\frac{1}{N}X\Sigma_NX^{\dag}.
\end{equation}
The matrix $\underline{B}_N$ has the same eigenvalues as $B_N$
together with an addition $M-N$ zero eigenvalues. Therefore the
e.d.f $\underline{F}_N$ of the eigenvalues of $\underline{B}_N$ are
related to $F_N$ by
\begin{equation}\label{eq:FNline}
\underline{F}_N=(1-c_N)I_{[0,\infty)}+c_NF_N,\quad c_N=\frac{N}{M}.
\end{equation}
where $I_{[0,\infty)}$ is the step function that is $0$ on
$\mathbb{R}_-$ and $1$ on $\mathbb{R}_+$. In particular, as
$N\rightarrow\infty$, the distribution $\underline{F}_N$ converges
weakly to a p.d.f $\underline{F}$ that is related to $F$ by
\begin{equation}\label{eq:Fline}
\underline{F}=(1-c)I_{[0,\infty)}+cF
\end{equation}
and their Stieltjes transforms are related by
\begin{equation}\label{eq:stielim}
m_{\underline{F}}(z)=-\frac{1-c}{z}+cm_F(z).
\end{equation}
Then from (\ref{eq:algeq}), we see that the Stieltjes transform
$m_{\underline{F}}(z)$ satisfies the following equation
\begin{equation}\label{eq:meq}
m_{\underline{F}}(z)=-\left(z-c\int_{\mathbb{R}}\frac{tdH(t)}{1+tm_{\underline{F}}}\right)^{-1}.
\end{equation}
This equation has an inverse \cite{BS98}, \cite{BS99}
\begin{equation}\label{eq:zeq}
z(m_{\underline{F}})=-\frac{1}{m_{\underline{F}}}+c\int_{\mathbb{R}}\frac{t}{1+tm_{\underline{F}}}dH(t).
\end{equation}
The points where $\frac{d z(\xi)}{d\xi}=0$ are of significant
interest to us as they are potential end points of the support of
$\underline{F}$, due to the following result by Choi and
Silverstein.
\begin{lemma}\label{le:CS}$\cite{CS}$ If $z\notin\mathrm{supp}(\underline{F})$,
then $m=m_{\underline{F}}(z)$ satisfies the following.
\begin{enumerate}
\item $m\in\mathbb{R}\setminus\{0\}$;
\item $-\frac{1}{m}\notin\mathrm{supp}(H)$;
\item Let $z$ be defined by (\ref{eq:zeq}), then $z^{\prime}(m)>0$, where the prime denotes
the derivative with respect to $m_{\underline{F}}$ in
(\ref{eq:zeq}).
\end{enumerate}
Conversely, if $m$ satisfies 1-3, then
$z=z(m)\notin\mathrm{supp}(\underline{F})$.
\end{lemma}
This lemma allows us to identify the complement of
$\mathrm{supp}(\underline{F})$ by studying the real points $m$ such
that $z^{\prime}(m)>0$.
\begin{remark}\label{re:baik}
As pointed out in \cite{Baikspike}, Lemma \ref{le:CS} applies to any
distribution $G$ whose Stieltjes transform $m_{G}(z)$
(\ref{eq:stie}) satisfies an equation of the form
\begin{equation*}
z(m_{G})=-\frac{1}{m_{G}}+c_G\int_{\mathbb{R}}\frac{t}{1+tm_{G}}dH_G(t).
\end{equation*}
for some constant $c_G$ and distribution $H_G(t)$.
\end{remark}

\subsection{Riemann surface and the Stieltjes transform}
We will now restrict ourselves to the case when the matrix
$\Sigma_N$ has 2 distinct eigenvalues only. Without lost of
generality, we will assume that one of these values is 1 and the
other one is $a>0$. Let $0<\beta<1$, we will assume that as
$N\rightarrow\infty$, $N_1$ of the eigenvalues take the value $a$
and $N_0=N-N_1$ of the eigenvalues are 1 and that
$\frac{N_1}{N}\rightarrow\beta$. That is, as $N\rightarrow\infty$,
the e.d.f $H_N(x)$ converges to the following
\begin{equation}\label{eq:limedf}
dH_N(x)\rightarrow dH(x)=(1-\beta)\delta_1+\beta\delta_{a}.
\end{equation}
By substituting this back into (\ref{eq:zeq}), we see that the
Stieltjes transform $\xi(z)=m_{\underline{F}}(z)$ is a solution to
the following algebraic equation
\begin{equation}\label{eq:curve1}
z=-\frac{1}{\xi}+c\frac{1-\beta}{1+\xi}+c\frac{a\beta}{1+a\xi}.
\end{equation}
Rearranging the terms, we see that $\xi=m_{\underline{F}}$ solves
the following
\begin{equation}\label{eq:curve2}
\begin{split}
za\xi^3&+(A_2z+B_2)\xi^2+(z+B_1)\xi+1=0,\\
A_2&=(1+a),\quad B_2=a(1-c),\\
B_1&=1-c(1-\beta)+a(1-c\beta).
\end{split}
\end{equation}
This defines a Riemann surface $\Lie$ as a 3-folded cover of the
complex plane.

By solving the cubic equation (\ref{eq:curve2}) or by analyzing the
asymptotic behavior of the equation as $z\rightarrow\infty$, we see
that these solutions have the behavior given by (\ref{eq:xiinfty})
as $z\rightarrow\infty$. On the other hand, as $z\rightarrow 0$, the
3 branches of $\xi(z)$ behave as follows
\begin{equation}\label{eq:zeroasym}
\begin{split}
\xi_{\alpha}(z)&=-\frac{1-c}{z}+O(1),\quad z\rightarrow 0,\\
\xi_{\beta}(z)&=R_1+O(z),\quad z\rightarrow 0,\\
\xi_{\gamma}(z)&=R_2+O(z),\quad z\rightarrow 0.
\end{split}
\end{equation}
where the order of the indices $\alpha$, $\beta$ and $\gamma$ does
not necessarily coincide with the ones in (\ref{eq:xiinfty}) (i.e.
we do not necessarily have $\alpha=1$, $\beta=2$ and $\gamma=3$).
The constants $R_1$ and $R_2$ are the two roots of the quadratic
equation
\begin{equation}\label{eq:quad}
a(1-c)x^2+(1-c(1-\beta)+a(1-c\beta))x+1=0.
\end{equation}
The functions $\xi_j(z)$ will not be analytic at the branch points
of $\Lie$ and they will be discontinuous across the branch cuts
joining these branch points. Moreover, from (\ref{eq:zeroasym}), one
of the functions $\xi_j(z)$ will have a simple pole at $z=0$. Apart
from these singularities, however, the functions $\xi_j(z)$ are
analytic.

\subsection{Sheet structure of the Riemann surface}
As explained in the introduction, the determination of the sheet
structure of the Riemann surface (\ref{eq:curve20}) involves
difficult analysis of the behavior of the $\xi_j(z)$ at the points
$\lambda_k$. In order to implement Riemann-Hilbert analysis to
obtain the asymptotics of the kernel in (\ref{eq:mpoint}), we must
show that $\xi_1(z)$ is not analytic at the real branch points
$\lambda_k$, $k=1,2$. This will be achieved by making use of the
properties of the Stieltjes transform.

In this section we will study the sheet structure of the Riemann
surface $\Lie$. As we shall see, the branch $\xi_1(z)$ turns out to
be the Stieljes transform $m_{\underline{F}}(z)$ and its branch cut
will be the support of $\underline{F}$.

By Lemma \ref{le:CS}, the real points at which $\frac{d z}{d\xi}>0$
characterize the end points of $\mathrm{Supp}(\underline{F})$. The
determine the support, let us differentiate (\ref{eq:curve1}) to
obtain an expression of $\frac{d z}{d\xi}$ in terms of $\xi$.
\begin{equation}\label{eq:zprime}
\begin{split}
\frac{d
z}{d\xi}&=\frac{1}{\xi^2(1+\xi)^2(1+a\xi)^2}\Bigg(a^2(1-c)\xi^4
+2(a^2(1-c\beta)+a(1-c(1-\beta))\xi^3\\
&+(1-c(1-\beta)+a^2(1-c\beta)+4a)\xi^2 +2(1+a)\xi+1\Bigg).
\end{split}
\end{equation}
In particular, the values of $\xi$ at $\frac{d z}{d\xi}=0$
correspond to the roots of the quartic equation
\begin{equation}\label{eq:quart}
\begin{split}
a^2(1-c)\xi^4
&+2(a^2(1-c\beta)+a(1-c(1-\beta))\xi^3\\
&+(1-c(1-\beta)+a^2(1-c\beta)+4a)\xi^2 +2(1+a)\xi+1=0
\end{split}
\end{equation}
Let $\Delta$ be the discriminant of this quartic polynomial, then
when $\Delta<0$, the equation (\ref{eq:quart}) has 2 distinct real
roots $\gamma_1<\gamma_2$ and 2 complex roots $\gamma_3$ and
$\gamma_4=\overline{\gamma}_3$. One can check that the coefficients
of (\ref{eq:quart}) are all positive and hence
$\gamma_1<\gamma_2<0$.

Let $\lambda_k$ be the corresponding points in the $z$-plane
\begin{equation}\label{eq:lambda}
\begin{split}
\lambda_k&=-\frac{1}{\gamma_k}+c\frac{1-\beta}{1+\gamma_k}+c\frac{a\beta}{1+a\gamma_k},\quad
k=1,\ldots,4.
\end{split}
\end{equation}
Then we have the following
\begin{lemma}\label{le:lambdadis} The points $\lambda_k$,
$k=1,\ldots,4$ are all distinct.
\end{lemma}
\begin{proof}
The derivative $\frac{d z(\xi)}{d\xi}$ has simple zeros at the
points $(\lambda_{k},\gamma_{k})$, $k=1,\ldots,4$, on the Riemann
surface defined by (\ref{eq:curve2}). Therefore, as a function of
$z$, 2 branches of the function $\xi(z)$ behaves as
\begin{equation*}
\xi(z)=\gamma_{k}\pm
C_{k}(z-\lambda_{k})^{\frac{1}{2}}+O(z-\lambda_{k}),\quad
z\rightarrow\lambda_{k},
\end{equation*}
near the points $\lambda_{k}$. Let $i\neq j$, then Since
$\gamma_i\neq\gamma_j$, if $\lambda_i=\lambda_j$, there will be 4
distinct solutions $\xi(z)$ to the equation (\ref{eq:curve2}) in a
neighborhood of the point $\lambda_{i}=\lambda_j$, which is not
possible. Therefore $\lambda_{i}$ and $\lambda_j$ are distinct.
\end{proof} In particular, the points $\lambda_{3}$ and $\lambda_4$
are complex.
\begin{lemma}\label{le:lambdapm}
The points $\lambda_{3}$ and $\lambda_4$ are not real.
\end{lemma}
\begin{proof} Since $\gamma_3=\overline{\gamma}_4$, we see that
$\lambda_3=\overline{\lambda}_4$. If $\lambda_{3}$ and $\lambda_4$
are real, then we will have $\lambda_3=\lambda_4$. This contradicts
Lemma \ref{le:lambdadis} and hence the points $\lambda_{3}$ and
$\lambda_4$ are not real.
\end{proof}
We will label the points $\gamma_3$ and $\gamma_4$ such that
$\mathrm{Im}(\lambda_3)>0$.

Note that, from the behavior of $z(\xi)$ in (\ref{eq:curve1}), we
see that near the points $-1$ and $-\frac{1}{a}$, the function
$z(\xi)$ behaves as
\begin{equation}\label{eq:zsing}
\begin{split}
z(\xi)&=\frac{c(1-\beta)}{1+\xi}+O(1),\quad \xi\rightarrow -1,\\
z(\xi)&=\frac{ca\beta}{1+a\xi}+O(1),\quad \xi\rightarrow
-\frac{1}{a}.
\end{split}
\end{equation}
and hence $z^{\prime}(\xi)$ is negative near these points. From this
and (\ref{eq:zprime}), we see that $z^{\prime}(\xi)>0$ on the
intervals $(-\infty,\gamma_1)$, $(\gamma_2,0)$ and $(0,\infty)$ and
none of the points $-1$ or $-\frac{1}{a}$ belongs to these
intervals. On $(\gamma_1,\gamma_2)$, $z^{\prime}(\xi)$ is negative.

The images of these intervals under the map (\ref{eq:curve1}) then
give us the complement of $\mathrm{supp}(\underline{F})$ in the
$z$-plane. Let us study these images
\begin{lemma}\label{le:image}
The intervals $(-\infty,\gamma_1)$, $(\gamma_2,0)$ and $(0,\infty)$
are mapped by $z(\xi)$ to $(0,\lambda_1)$, $(\lambda_2,\infty)$ and
$(-\infty,0)$ respectively.
\end{lemma}
\begin{proof} Since none of the points $-1$, $-\frac{1}{a}$ and $0$
belongs to these intervals both $z(\xi)$ and $z^{\prime}(\xi)$ are
continuous on these intervals. Moreover, $z(\xi)$ is strictly
increasing on these intervals. Therefore the images of these
intervals are given by
\begin{equation*}
\begin{split}
z\left((-\infty,\gamma_1)\right)&=(z(-\infty),z(\gamma_1))=(0,\lambda_1)\\
z\left((\gamma_2,0)\right)&=(z(\gamma_2),z(0^-))=(\lambda_2,\infty)\\
z\left((0,\infty)\right)&=(z(0^+),z(\infty))=(-\infty,0).
\end{split}
\end{equation*}
where the $\pm$ superscripts in the above indicates that the
function is evaluated at $\pm\epsilon$ for $\epsilon\rightarrow 0$.
\end{proof}
Therefore the complement of $\mathrm{supp}(\underline{F})$ is given
by (recall that $\underline{F}$ has a point mass at $0$)
\begin{equation}\label{eq:suppc}
\mathrm{supp}(\underline{F})^c=(-\infty,0)\cup(0,\lambda_1)\cup(\lambda_2,\infty).
\end{equation}
Let us now show that $\lambda_1<\lambda_2$. This would imply the
support of $\underline{F}$ is non-empty and consists of one
interval.
\begin{lemma}\label{le:1cut} Let $\lambda_1$ and $\lambda_2$ be the points defined
by (\ref{eq:lambda}), then $\lambda_1<\lambda_2$ and hence the
support of $\underline{F}$ consists of a single interval.
\end{lemma}
\begin{proof} Suppose $\lambda_1>\lambda_2$. Let
$z_0\in(\lambda_2,\lambda_1)$. Since both the points $-1$ and
$-\frac{1}{a}$ belongs to $(\gamma_1,\gamma_2)$, the function
$z(\xi)$ is continuous in $(-\infty,\gamma_1)$ and $(\gamma_2,0)$.
By Lemma \ref{le:image}, these two intervals are mapped onto
$(0,\lambda_1)$ and $(\lambda_2,\infty)$ by the map $z(\xi)$. Hence
there is at least one point in each of the intervals
$(-\infty,\gamma_1)$ and $(\gamma_2,0)$ that is being mapped onto
$z_0$. On the other hand, let $s_1$ and $s_2$ be the points in
$\{-1,-\frac{1}{a}\}$ that is closer to $\gamma_1$ and $\gamma_2$
respectively, then $z(\xi)$ is continuous in $(\gamma_1,s_1)$ and
$(s_2,\gamma_2)$. Since these intervals are mapped onto
$(-\infty,\lambda_1,)$ and $(\lambda_2,\infty)$ respectively by
$z(\xi)$, there is at least one point in each of these intervals
that is mapped onto $z_0$. This would result in 4 distinct points
being mapped onto $z_0$ by the map $z(\xi)$, which is not possible.
Therefore we have $\lambda_1<\lambda_2$.
\end{proof}
Therefore we have the following
\begin{theorem}\label{thm:cuts}
Let $\Delta$ be the discriminant of the quartic polynomial
\begin{equation}\label{eq:quartic}
\begin{split}
a^2(1-c)\xi^4
&+2(a^2(1-c\beta)+a(1-c(1-\beta))\xi^3\\
&+(1-c(1-\beta)+a^2(1-c\beta)+4a)\xi^2 +2(1+a)\xi+1=0.
\end{split}
\end{equation}
then if $\Delta<0$, the support of $\underline{F}$ consists of a
single interval.
\end{theorem}
We can treat (\ref{eq:curve2}) as a polynomial in $\xi$ then
$\lambda_k$, $k=1,\ldots,4$ are the zeroes of its discriminant
$D_3(z)=(az)^4\prod_{i<j}(\xi_i-\xi_j)^2$.
\begin{equation}\label{eq:D3}
\begin{split}
D_3(z)&=(1-a)^2z^4+(2A_2^2B_1+2A_2B_2-4A_2^3-12aB_1+18aA_2)z^3\\
&+(B_2^2+A_2^2B_1^2+4A_2B_1B_2-12A_2^2B_2-12aB_1^2+18aB_2+18aA_2B_1-27a^2)z^2\\
&+(2B_1B_2^2+2A_2B_2B_1^2-12A_2B_2^2-4B_1^3a+18aB_1B_2)z+B_1^2B_2^2-4B_2^3.
\end{split}
\end{equation}
The zeros of (\ref{eq:D3}) then correspond to the branch points of
the Riemann surface $\Lie$. These branch points are given on $\Lie$
by $(\lambda_k,\gamma_k)$, for $k=1,\ldots,4$.

Since the leading coefficient of $D_3(z)$ is $(1-a)^2>0$, we see
that the sign of $D_3(z)$ and hence the 3 roots of the cubic
polynomial (\ref{eq:curve2}) behave as follows for $z\in\mathbb{R}$.
\begin{equation}\label{eq:realim}
\begin{split}
&1. \quad z\in\mathbb{R}\setminus [\lambda_1,\lambda_2],\quad
D_3(z)>0\Rightarrow
\textrm{$\xi$ has 3 distinct real roots} \\
&2. \quad z\in (\lambda_1,\lambda_2),\quad
D_3(z)<0\Rightarrow\textrm{$\xi$ has 1 real and 2 complex roots}.
\end{split}
\end{equation}
In particular, since the roots coincide at the branch points, the
$\gamma_k$, $k=1,\ldots,4$ is the values of the double root of the
cubic (\ref{eq:curve2}) when $z=\lambda_k$.

We can now compute the probability density $\underline{F}$.
\begin{theorem}\label{thm:density}
Suppose $\Delta<0$. Then the p.d.f $\underline{F}$ is supported on
$[\lambda_1,\lambda_2]$ with the following density
$d\underline{F}(z)=\rho(z)dz$
\begin{equation}\label{eq:density}
\frac{3}{2\pi}\left|\left(\frac{r(z)+\sqrt{-\frac{1}{27a^4z^4}D_3(z)}}{2}\right)^{\frac{1}{3}}-\left(\frac{r(z)-\sqrt{-\frac{1}{27a^4z^4}D_3(z)}}{2}\right)^{\frac{1}{3}}\right|,
\end{equation}
where $D_3(z)$ is given by (\ref{eq:D3}) and $r(z)$ is given by
\begin{equation*}
\begin{split}
r(z)&=\frac{1}{27}\Bigg(-\frac{2B_2^3}{a^3}z^{-3}+\left(\frac{9B_1B_2}{a^2}-\frac{6A_2B_2^2}{a^3}\right)z^{-2}+\left(\frac{9B_2}{a^2}+\frac{9B_1A_2}{a^2}-\frac{27}{a}-\frac{6A_2^2B_2}{a^3}\right)z^{-1}\\
&+\left(\frac{9A_2}{a^2}-\frac{2A_2^3}{a^3}\right)\Bigg).
\end{split}
\end{equation*}
The cube root in (\ref{eq:density}) is chosen such that
$\sqrt[3]{A}\in\mathbb{R}$ for $A\in\mathbb{R}$ and the square root
is chosen such that $\sqrt{A}>0$ for $A>0$.
\end{theorem}
The proof of this theorem is the same as Theorem 6 in \cite{Mo}.

Let us now show that the function $\xi_1(z)$ is in fact the
Stieltjes transform $m_{\underline{F}}(z)$. From the asymptotic
behavior of the $\xi_j(z)$ (\ref{eq:xiinfty}) and the fact that the
Stieltjes transform $m_{\underline{F}}(z)$ is the unique solution of
(\ref{eq:curve2}) that vanishes as $z\rightarrow\infty$, we see that
\begin{equation}\label{eq:stiexi}
m_{\underline{F}}(z)=\xi_1(z).
\end{equation}
For $c<1$, it was shown in \cite{CS} that $F$ has a continuous
density and hence the Stieltjes transform $m_F(z)$ does not have any
poles. Therefore, by (\ref{eq:stielim}), we see that
$m_{\underline{F}}(z)$, and hence $\xi_1(z)$, has the following
singularity at $z=0$.
\begin{equation}\label{eq:xizero}
\xi_1(z)=-\frac{1-c}{z}+O(1),\quad z\rightarrow 0.
\end{equation}
Since $m_{\underline{F}}(z)$ is the Stieltjes transform of a measure
supported on the real axis, it is analytic away from the real axis
and hence by (\ref{eq:stiexi}), $\xi_1(z)$ is analytic at the points
$\lambda_3$ and $\lambda_4$ with a branch cut on
$[\lambda_1,\lambda_2]$. Since $\lambda_3$, $\lambda_4$ are not
branch points of the function $\xi_1(z)$, they must be branch points
of the functions $\xi_2(z)$ and $\xi_3(z)$. This determines the
branch structure of the Riemann surface $\Lie$. The branch cut of
$\xi_2(z)$ and $\xi_3(z)$ will eventually be chosen to be a contour
that goes from $\lambda_4$ to $\lambda_3$ intersecting the real axis
at a point in $(\lambda_1,\lambda_2)$, but in the next section it
will be chosen in a few different ways according to the situation.

\subsubsection{Geometry of the problem}\label{se:geo}
As explained in the introduction, the determination of the zero set
of $h(x)$ in (\ref{eq:hx}) is considerably more difficult than it is
in \cite{BKext2} and \cite{Lysov}. In this section we will carry out
a thorough analysis of this set and determine its topology.

For the implementation of the Riemann-Hilbert analysis, it is more
convenient to consider a measure $\hat{F}_{N}$ instead of the
measure $\underline{F}$.

Let $c_N=\frac{N}{M}$ and $H_N(t)$ be the e.d.f
\begin{equation*}
dH_N(t)=(1-\beta_N)\delta_1+\beta_N\delta_a.
\end{equation*}
where $\beta_N=\frac{N_1}{N}$.

Then $\hat{F}_N$ is the measure whose Stieltjes transform
$m_N(z)=m_{\hat{F}_N}(z)$ is the unique solution of
\begin{equation}\label{eq:mN}
z(m_N)=-\frac{1}{m_N}+c_N\int_{\mathbb{R}}\frac{t}{1+tm_N}dH_N(t),
\end{equation}
in $\mathbb{C}^+$ that behaves like $-\frac{1}{z}$ as
$z\rightarrow\infty$. Note that, as pointed out in \cite{Baikspike},
the measure $\hat{F}_N$ is not the eigenvalue distribution for
finite $N$, instead, it is only defined through the equation
(\ref{eq:mN}). From (\ref{eq:mN}), we see that $m_N(z)$ is the
solution of the algebraic equation
\begin{equation}\label{eq:curveN}
\begin{split}
za\xi^3&+(A_2z+B_2^N)\xi^2+(z+B_1^N)\xi+1=0, \\
B_2^N&=a(1-c_N),\quad B_1^N=1-c_N(1-\beta_N)+a(1-c_N\beta_N).
\end{split}
\end{equation}
that behaves as $-\frac{1}{z}$ as $z\rightarrow\infty$. If we denote
by $\xi_j^N(z)$ the solutions of (\ref{eq:curveN}) with asymptotic
behavior (\ref{eq:xiinfty}), but with $c$ and $\beta$ replaced by
$c_N$ and $\beta_N$, then we have $\xi_1^N(z)=m_N(z)$. In \cite{CS},
it was shown that the measure $\hat{F}_N$ has a continuous density
on $\mathbb{R}_+$ and a point mass of size $1-c_N$ at $0$ and hence
$\xi_1^N(z)$ also have the asymptotic behavior (\ref{eq:xizero})
near $z=0$ with $c$ replaced by $c_N$. By Remark \ref{re:baik}, all
the results in the previous section will remain valid for
(\ref{eq:curveN}) and $\xi_j^N(z)$, with $\beta$ and $c$ replaced by
$\beta_N$ and $c_N$. We will denote the Riemann surface defined by
(\ref{eq:curveN}) $\Lie_N$.

Let $\Delta_N$ be the determinant of the quartic polynomial
(\ref{eq:quart})
\begin{equation}\label{eq:quartN}
\begin{split}
a^2(1-c_N)\xi^4
&+2(a^2(1-c_N\beta_N)+a(1-c_N(1-\beta_N))\xi^3\\
&+(1-c_N(1-\beta_N)+a^2(1-c_N\beta_N)+4a)\xi^2 +2(1+a)\xi+1=0
\end{split}
\end{equation}
Since the determinant $\Delta$ of the quartic polynomial
(\ref{eq:quart}) is continuous in the parameters $a$, $\beta$ and
$c$, for large enough $N$, $M$ and $N_1$, we can assume that the
determinant $\Delta_N<0$. Then let $\gamma_1^N<\gamma_2^N$ be the
real roots of (\ref{eq:quartN}) and $\gamma_{3}^N$ and $\gamma_4^N$
be the complex conjugate roots of (\ref{eq:quartN}) and let
$\lambda_1^N<\lambda_2^N$ and $\lambda_{3}^N$, $\lambda_4^N$ be
their images under the map $z(\xi)$ given in (\ref{eq:mN}).

For the time being, we will choose the branch cut of $\xi_1^N(z)$ to
be the interval $[\lambda_1^N,\lambda_2^N]$, and the branch cut
between $\lambda_{3}^N$ and $\lambda_4^N$ to be a simple contour
$\mathcal{C}$ that is symmetric with respect to the real axis and
oriented upwards. Across $\mathcal{C}$, the two branches
$\xi_2^N(z)$ and $\xi_3^N(z)$ change into each other. The branch cut
$\mathcal{C}$ will be chosen such that it intersects the real axis
at exactly one point $x^{\ast}$ that is not equal to $\lambda_1^N$
or $\lambda_2^N$.

We will now define the functions $\theta_j^N(z)$ to be the the
integrals of $\xi_j^N(z)$.
\begin{equation}\label{eq:theta}
\begin{split}
\theta_1^N(z)&=\int_{\lambda_l^N}^z\xi_1^N(x)dx,\quad
\theta_2^N(z)=\int_{\lambda_{3}^N}^z\xi_2^N(x)dx,\quad
\theta_3^N(z)=\int_{\lambda_{3}^N}^z\xi_3^N(x)dx,\\
\textrm{if }x^{\ast}&<\lambda_2^N,\quad l=2;\quad \textrm{if
$x^{\ast}>\lambda_2^N,\quad l=1.$}
\end{split}
\end{equation}
The integration paths of the above integrals are chosen as follows.
If $x^{\ast}<\lambda_2^N$, then the integration path will not
intersect the set $\mathcal{C}\cup(-\infty,\lambda_2^N)$ and if
$x^{\ast}>\lambda_2^N$, then the integration path will not intersect
the set $\mathcal{C}\cup(\lambda_1^N,\infty)$.

Then from (\ref{eq:xiinfty}), (\ref{eq:zeroasym}) and
(\ref{eq:xizero}), we see that the integrals (\ref{eq:theta}) have
the following behavior at $z=\infty$ and $z=0$.
\begin{equation}\label{eq:asymtheta}
\begin{split}
\theta_1^N(z)&=-\log z+l_1^N+O\left(z^{-1}\right),\quad
z\rightarrow\infty,\\
\theta_1^N(z)&=-(1-c_N)\log z+O\left(1\right),\quad
z\rightarrow 0,\\
\theta_2^N(z)&=-z+c_N(1-\beta_N)\log
z+l_2^N+O\left(z^{-1}\right),\quad
z\rightarrow\infty,\\
\theta_2^N(z)&=O\left(1\right),\quad
z\rightarrow 0,\\
\theta_3^N(z)&=-\frac{z}{a}+c_N\beta_N\log
z+l_3^N+O\left(z^{-1}\right),\quad
z\rightarrow\infty,\\
\theta_3^N(z)&=O\left(1\right),\quad z\rightarrow 0.
\end{split}
\end{equation}
for some constants $l_1^N$, $l_2^N$ and $l_3^N$.

Then the set $\mathfrak{H}$ defined by
\begin{equation}\label{eq:frakH}
\mathfrak{H}=\left\{z\in\mathbb{C}|\quad
\mathrm{Re}\left(\theta_2^N(z)-\theta_3^N(z)\right)=0\right\}
\end{equation}
is important to the Riemann-Hilbert analysis. Let us now study its
properties. First note that it is symmetric across the real axis.
\begin{lemma}\label{le:symm}
Let the branch cut $\mathcal{C}$ between $\lambda_{3}^N$ and
$\lambda_4^N$ be a simple contour that is symmetric with respect to
the real axis. Then the set $\mathfrak{H}$ in (\ref{eq:frakH}) is
symmetric with respect to the real axis.
\end{lemma}
\begin{proof} Let us consider the transformation
$z\mapsto\overline{z}$ and
$\xi(z)\mapsto\overline{\xi}(\overline{z})$ in (\ref{eq:curveN}).
This gives
\begin{equation}\label{eq:curveNcon}
\begin{split}
\overline{z}a\overline{\xi}(\overline{z})^3&+(A_2\overline{z}+B_2^N)\overline{\xi}(\overline{z})^2+(\overline{z}+B_1^N)\overline{\xi}(\overline{z})+1=0.
\end{split}
\end{equation}
This means that if $\xi(z)$ is a solution to (\ref{eq:curveN}), then
so is $\overline{\xi}(\overline{z})$. Since the functions
$\overline{\xi}_j^N(\overline{z})$ is analytic away from
$\mathcal{C}\cup[\lambda_1^N,\lambda_2^N]$, it must equal either of
the functions $\xi_1^N(z)$, $\xi_2^N(z)$ or $\xi_3^N(z)$. By
considering the behavior of $\overline{\xi}_j^N(\overline{z})$ near
$z=\infty$ using (\ref{eq:xiinfty}), we see that
$\overline{\xi}_j^N(\overline{z})=\xi_j^N(z)$ for $j=1,2,3$. Let
$z_0\in\mathfrak{H}$, then we have
\begin{equation}\label{eq:zoH}
\mathrm{Re}\left(\int_{\lambda^N_3}^{z_0}(\xi_2^N(x)-\xi_3^N(x))dx\right)=0.
\end{equation}
By taking the complex conjugation of (\ref{eq:zoH}) and making use
of the fact that $\overline{\xi}_j^N(\overline{z})=\xi_j^N(z)$ for
$j=2,3$, we obtain
\begin{equation}
\mathrm{Re}\left(\int_{\lambda^N_4}^{\overline{z}_0}(\xi_2^N(x)-\xi_3^N(x))dx\right)=0.
\end{equation}
Let us now show that
$\mathrm{Re}\left(\theta_2^N(\lambda_4^N)-\theta_3^N(\lambda_4^N)\right)=0$.
Consider an integration contour $\Gamma$ from $\lambda^N_3$ to
$\lambda_4^N$ that is symmetric with respect to $\mathbb{R}$. Then
we have
\begin{equation*}
\begin{split}
\overline{\mathrm{Re}\left(\int_{\lambda_3^N}^{\lambda_4^N}\left(\xi_2^N(x)-\xi_3^N(x)\right)dx\right)}&=
\mathrm{Re}\left(\int_{\lambda_4^N}^{\lambda_3^N}\left(\xi_2^N(x)-\xi_3^N(x)\right)dx\right)\\
&=-\mathrm{Re}\left(\int_{\lambda_3^N}^{\lambda_4^N}\left(\xi_2^N(x)-\xi_3^N(x)\right)dx\right),
\end{split}
\end{equation*}
where we have used $\lambda_3^N=\overline{\lambda}_4^N$ in the
above. Hence we have
$\mathrm{Re}\left(\theta_2^N(\lambda_4^N)-\theta_3^N(\lambda_4^N)\right)=0$.
This implies the following
\begin{equation}
\mathrm{Re}\left(\int_{\lambda_4^N}^{\overline{z}_0}(\xi_2^N(x)-\xi_3^N(x))dx\right)=
\mathrm{Re}\left(\int_{\lambda_3^N}^{\overline{z}_0}(\xi_2^N(x)-\xi_3^N(x))dx\right)=0.
\end{equation}
Therefore if $z_0\in\mathfrak{H}$, then $\overline{z}_0$ is also in
$\mathfrak{H}$ and hence $\mathfrak{H}$ is symmetric with respect to
the real axis.
\end{proof}
We will now show that the set $\mathfrak{H}$ is independent on the
choice of the branch cut $\mathcal{C}$.
\begin{lemma}\label{le:indep}
Let the branch cut $\mathcal{C}$ be symmetric with respect to the
real axis. Then the set $\mathfrak{H}$ is independent on the choice
of the branch cut $\mathcal{C}$.
\end{lemma}
\begin{proof} Let us consider the boundary values of $\theta_2^N(z)$
and $\theta_3^N(z)$ along the branch cut $\mathcal{C}$. Let
$\xi_{2,\pm}^N(z)$ and $\xi_{3,\pm}^N(z)$ be the boundary values of
$\xi_2^N(z)$ and $\xi_3^N(z)$ on the left and right hand sides of
$\mathcal{C}$, then we have $\xi_{2,\pm}^N(z)=\xi_{3,\mp}^N(z)$ for
$z\in\mathcal{C}$. Therefore, for $z$ in the upper half plane, we
have the following
\begin{equation*}
\int_{\lambda_3^N}^{z}(\xi_{2,+}^N(x)-\xi_{3,+}^N(x))dx=-\int_{\lambda_3^N}^{z}(\xi_{2,-}^N(x)-\xi_{3,-}^N(x))dx
\end{equation*}
where the integration is performed along $\mathcal{C}$. Therefore in
the upper half plane, the function $\theta_2^N(z)-\theta_3^N(z)$
changes sign across the branch cut $\mathcal{C}$. Hence in the upper
half plane, the zero set of
$\mathrm{Re}\left(\theta_2^N(z)-\theta_3^N(z)\right)$ is independent
on the choice of $\mathcal{C}$. Now by Lemma \ref{le:symm} the set
$\mathfrak{H}$ is symmetric with respect to the real axis for any
choice of symmetric branch cut $\mathcal{C}$. Therefore the set
$\mathfrak{H}$ in the lower half plane is just the reflection of
$\mathfrak{H}$ in the upper half plane which must also be
independent on the choice of $\mathcal{C}$.
\end{proof}
We will now show that the real parts of $\xi_2^N(z)$ and
$\xi_3^N(z)$ will coincide exactly once on the real axis.
\begin{lemma}\label{le:real23}
Let $x^{\ast}$ be the intersection between $\mathcal{C}$ and
$\mathbb{R}$, then the real function
$\mathrm{Re}\left(\xi_2^N(z)-\xi_3^N(z)\right)$ is continuous on
$(-\infty,x^{\ast})$ and $(x^{\ast},\infty)$ and it vanishes exactly
once at a point $\iota\in\mathbb{R}\setminus\{x^{\ast}\}$.
\end{lemma}
\begin{proof} The location
of $x^{\ast}$ is immaterial as the function
$\mathrm{Re}\left(\xi_2^N(z)-\xi_3^N(z)\right)$ only changes sign
across the point $x^{\ast}$ and hence its zeros on $\mathbb{R}$ are
independent on the location of $x^{\ast}$. For definiteness, let us
assume that $x^{\ast}<\lambda_1^N$.

Let $l_1=2$, $l_2=3$ for $a>1$ and $l_1=3$, $l_2=2$ for $a<1$. From
the behavior (\ref{eq:xiinfty}) of $\xi_j^N(z)$ near $z=\pm\infty$,
we see that there exists $R>0$ such that $\xi_1^N(\pm
R)>\xi_{l_2}^N(\pm R)>\xi_{l_1}^N(\pm R)$. Note that
$\xi_{l_1}^N(z)$ and $\xi_{l_2}^N(z)$ are continuous on
$(-\infty,x^{\ast})\cup(\lambda_2^N,\infty)$. Since these intervals
do not contain any branch point of (\ref{eq:curveN}) and both
$\xi_{l_1}^N(z)$ and $\xi_{l_2}^N(z)$ are real on them,
$\mathrm{Re}\left(\xi_{l_1}^N(z)\right)$ and
$\mathrm{Re}\left(\xi_{l_2}^N(z)\right)$ cannot coincide on these
intervals. In particular, the order $\xi_{l_2}^N(z)>\xi_{l_1}^N(z)$
must be preserved in $(-\infty,x^{\ast})\cup(\lambda_2^N,\infty)$.
Since $\xi_{l_1}^N(z)$ and $\xi_{l_2}^N(z)$ interchange across the
branch cut $\mathcal{C}$, we must have
$\xi_{l_1}^N(z)>\xi_{l_2}^N(z)$ in the interval
$(x^{\ast},\lambda_1^N)$. This means that
$\xi_{l_1}^N(z)>\xi_{l_2}^N(z)$ on the left hand side of
$[\lambda_1^N,\lambda_2^N]$, while $\xi_{l_1}^N(z)<\xi_{l_2}^N(z)$
on the right hand side of $[\lambda_1^N,\lambda_2^N]$ and hence
$\mathrm{Re}\left(\xi_2^N(z)\right)$ and
$\mathrm{Re}\left(\xi_3^N(z)\right)$ must coincide at least once in
$[\lambda_1^N,\lambda_2^N]$. We will show that they can only
coincide once within $[\lambda_1^N,\lambda_2^N]$.

Inside $[\lambda_1^N,\lambda_2^N]$, the function $\xi_1^N(z)$ and
one other root $\xi_I^N(z)$ becomes complex and are conjugate to
each other, while another root $\xi_R^N(z)$ remains real. Note that,
since $\xi_{I,\pm}^N(z)=\xi_{1,\mp}^N(z)$ on the branch cut
$[\lambda_1^N,\lambda_2^N]$, we see that the real part of
$\xi_I^N(z)$ and $\xi_1^N(z)$ is continuous across
$[\lambda_1^N,\lambda_2^N]$, while $\xi_R^N(z)$ has no jump
discontinuity across $[\lambda_1^N,\lambda_2^N]$. Since neither
$\xi_I^N(z)$ or $\xi_R^N(z)$ are equal to the branch $\xi_1^N(z)$ in
$(\lambda_1^N,\lambda_2^N)$, the real function $\mathrm{Re}\left(
\xi_I^N(z)-\xi_R^N(z)\right)$ must equal either of
$\pm\mathrm{Re}\left( \xi_2^N(z)-\xi_3^N(z)\right)$ and hence the
real function $\mathrm{Re}\left( \xi_2^N(z)-\xi_3^N(z)\right)$ does
not have jump discontinuity on $[\lambda_1^N,\lambda_2^N]$.
Therefore for $z\in\mathbb{R}$, this real function can only have
jump discontinuity at the point $x^{\ast}$.

From the coefficient of $\xi^2$ in (\ref{eq:curveN}), we see that
\begin{equation}\label{eq:coef}
2\mathrm{Re}\left(\xi_I^N(z)\right)+\xi_R^N(z)=-\frac{1}{a}\left(A_2+\frac{B_2^N}{z}\right).
\end{equation}
Taking the derivative with respect to $z$, we obtain
\begin{equation*}
2\frac{d}{dz}\mathrm{Re}\left(
\xi_I^N(z)-\xi_R^N(z)\right)+3\frac{d\xi_R^N(z)}{dz}=\frac{1}{a}\frac{B_2^N}{z^2}.
\end{equation*}
This implies
\begin{equation*}
3\frac{d\xi_R^N(z)}{dz}=\frac{1}{a}\frac{B_2^N}{z^2}-2\frac{d}{dz}\mathrm{Re}\left(
\xi_I^N(z)-\xi_R(z)\right).
\end{equation*}
From (\ref{eq:curveN}), it is easy to see that $B_2^N>0$ as $c_N<1$.
Hence if the derivative of $\mathrm{Re}\left(
\xi_I^N(z)-\xi_R(z)\right)$ is non-positive at a point
$z_0\in[\lambda_1^N,\lambda_2^N]$, then we will have
$\frac{d\xi_R(z_0)}{dz}>0$. Since $\xi_R(z)$ is real, this would
imply the derivative $\frac{d z(\xi)}{d\xi}$ of the function $z$
defined by (\ref{eq:mN}) is positive at the real point
$m=\xi_R(z_0)$. By Lemma \ref{le:CS}, the point $z_0=z(m)$ cannot
belong to
$\mathrm{Supp}\left(\hat{F}_N\right)=[\lambda_1^N,\lambda_2^N]$.
This leads to a contradiction and hence we must have
\begin{equation}\label{eq:posde}
\frac{d}{dz}\mathrm{Re}\left( \xi_I^N(z)-\xi_R^N(z)\right)>0,\quad
z\in[\lambda_1^N,\lambda_2^N].
\end{equation}
In particular, if the function $\mathrm{Re}\left(
\xi_I^N(z)-\xi_R^N(z)\right)$ has more than one zero inside of
$[\lambda_1^N,\lambda_2^N]$, then at one of the zeros, the
derivative in (\ref{eq:posde}) must be smaller than or equal to
zero. This is a contradiction and hence the function
$\mathrm{Re}\left( \xi_I^N(z)-\xi_R^N(z)\right)$ can vanish at most
once inside $[\lambda_1^N,\lambda_2^N]$. This means that the
function $\mathrm{Re}\left( \xi_2^N(z)-\xi_3^N(z)\right)$ will also
vanish at most once inside $[\lambda_1^N,\lambda_2^N]$. Since we
have already shown that $\mathrm{Re}\left(
\xi_2^N(z)-\xi_3^N(z)\right)$ vanishes at least once inside this
interval, it must then vanish exactly once inside
$[\lambda_1^N,\lambda_2^N]$. This concludes the proof of the lemma.
\end{proof}
We can now determine the number of intersection points between
$\mathfrak{H}$ and $\mathbb{R}$.
\begin{lemma}\label{le:inter}
The set $\mathfrak{H}$ intersects $\mathbb{R}$ at most twice.
\end{lemma}
\begin{proof} By Lemma \ref{le:real23}, there exists exactly one point
$\iota\in[\lambda_1^N,\lambda_2^N]$ such that
$\mathrm{Re}\left(\xi_2^N(\iota)-\xi_3^N(\iota)\right)=0$. Let us
assume that $\mathfrak{H}$ and $\mathbb{R}$ intersects at a point
$s_0<\iota$. The case when $s_0>\iota$ can be treated similarly. Let
us also choose $\mathcal{C}$ such that $\mathcal{C}$ intersects
$\mathbb{R}$ at $x^{\ast}<s_0$. As both $\mathfrak{H}$ and $\iota$
are independent on the choice of $\mathcal{C}$, a different choice
will not affect the intersection between $\mathfrak{H}$ and
$\mathbb{R}$. Since $x^{\ast}<s_0$, the function
$\mathrm{Re}\left(\xi_2^N-\xi_3^N\right)$ is continuous and has
different signs on $(s_0,\iota)$ and $(\iota,\infty)$. However,
within these two intervals, the sign of
$\mathrm{Re}\left(\xi_2^N-\xi_3^N\right)$ remains unchanged. Hence
there is at most one point $z\in(s_0,\infty)$ such that
$\mathrm{Re}\left(\int_{s_0}^{z}\xi_{2}^N-\xi_{3}^Ndx\right)=0$,
where the integration path is taken along $\mathbb{R}$. Since
$\mathrm{Re}\left(\int_{\lambda_3^N}^{s_0}\xi_{2}^N-\xi_{3}^Ndx\right)=0$
and that $\mathfrak{H}$ is symmetric with respect to the real axis,
we see that there is at most one point on $(s_0,\infty)$ that
belongs to $\mathfrak{H}$.

Let us now consider the possible intersection points on
$(-\infty,s_0)$. As the choice of the branch cut $\mathcal{C}$ does
not affect the set $\mathfrak{H}$ and its intersection with the real
axis, let us choose $\mathcal{C}$ so that it intersects $\mathbb{R}$
at a point $\tilde{x}_0>s_0$ instead. Then
$\mathrm{Re}\left(\xi_2^N-\xi_3^N\right)$ is continuous on
$(-\infty,s_0)$. Moreover, it does not change sign in
$(-\infty,s_0)$. Therefore the function
$\mathrm{Re}\left(\int_{\lambda_3^N}^{z}\xi_{2}^N-\xi_{3}^Ndx\right)$
does not vanish in $(-\infty,s_0)$ and hence $\mathfrak{H}$ does not
intersect $(-\infty,s_0)$. This shows that if $\mathfrak{H}$
intersects $\mathbb{R}$ at a point $s_0<\iota$, then there can at
most be 2 intersection points between $\mathfrak{H}$ and
$\mathbb{R}$. By using similar argument, one can show the same for
the case when $s_0>\iota$.

Let us now consider the case when $s_0=\iota$. If $s_0=\iota$, then
by choosing $\mathcal{C}$ such that $x^{\ast}<\iota$
($x^{\ast}>\iota$), we see that
$\mathrm{Re}\left(\xi_2^N-\xi_3^N\right)$ does not change sign in
$(s_0,\infty)$ ($(-\infty,s_0)$) and hence
$\mathrm{Re}\left(\int_{\lambda_3^N}^{z}\xi_{2}^N-\xi_{3}^Ndx\right)$
does not vanish in either of these intervals. Therefore
$\mathfrak{H}$ can only intersect $\mathbb{R}$ at the point
$s_0=\iota$. In any case, the set $\mathfrak{H}$ can intersect
$\mathbb{R}$ at 2 points at most.
\end{proof}
We can now determine the shape of the set $\mathfrak{H}$.
\begin{proposition}\label{pro:shapeH}
The set $\mathfrak{H}$ consists of 4 simple curves,
$\mathfrak{H}_{\infty}^{\pm}$, $\mathfrak{H}_L$ and
$\mathfrak{H}_R$. The curve $\mathfrak{H}_{\infty}^{+}$
($\mathfrak{H}_{\infty}^-$) is an open smooth curve that go from
$\lambda_{3}^N$ ($\lambda_4^N$) to infinity. They approach infinity
in a direction parallel to the imaginary axis and do not intersect
the real axis. The curves $\mathfrak{H}_L$ and $\mathfrak{H}_R$ are
simple curves joining $\lambda_3^N$ and $\lambda_4^N$. The curve
$\mathfrak{H}_L$ is in the left hand side of $\mathfrak{H}_R$ in the
complex plane and each of these curves intersects the real axis
once. These two curves are smooth except at their intersections with
real axis. Let $x_L$ and $x_R$ be the intersection points of
$\mathfrak{H}_L$ and $\mathfrak{H}_R$ with $\mathbb{R}$, then
$(x_L,x_R)\cap[\lambda_1^N,\lambda_2^N]\neq\emptyset$.
\end{proposition}
\begin{proof} Let the sets $\mathfrak{H}_+$ and $\mathfrak{H}_-$ be
the intersections of $\mathfrak{H}$ with the upper and lower half
planes respectively. Then these 2 sets are reflections of each other
with respect to the real axis. Within the set $\mathfrak{H}_+$,
there are 3 curves $\mathfrak{H}_0^+$, $\mathfrak{H}_1^+$ and
$\mathfrak{H}_2^+$ coming out of the point $\lambda_3^N$. Let us
show that these curves are smooth except at $\lambda_3^N$. Suppose
there is a point $z_0$ on $\mathfrak{H}_j^+$ that is not smooth.
This means that the function $\theta_2^N-\theta_3^N$ is not
conformal at $z_0$. Since $\mathfrak{H}$ is independent on the
choice of the branch cut $\mathcal{C}$, by changing the branch cut
if necessary, we can assume that both $\theta_2^N(z)$ and
$\theta_3^N(z)$ are analytic at $z_0$ and therefore the derivative
of $\theta_2^N-\theta_3^N$ must be zero at $z_0$ as the function is
not conformal at $z_0$. This would imply
$\xi_2^N(z_0)=\xi_3^N(z_0)$, which is impossible as the only points
where this happens are the points $\lambda_k^N$. Therefore the
curves $\mathfrak{H}_j^+$, $j=0,1,2$ are smooth except at the point
$\lambda_3^N$.

We will now show that the curves $\mathfrak{H}_j^+$ cannot be
connected with one another except at the point $\lambda_3^N$.
Suppose the curves $\mathfrak{H}_j^+$ is connected to
$\mathfrak{H}_k^+$ at a point $z_0\neq \lambda_3^N$. Since both
curves $\mathfrak{H}_j^+$ and $\mathfrak{H}_k^+$ are smooth, the
curve $\mathfrak{H}_j^+\cup\mathfrak{H}_k^+$ forms a close loop in
the upper half plane. Let $V$ be the region bounded by this close
loop. Then by changing the choice of $\mathcal{C}$ if necessary, we
can assume that the functions $\theta_2^N(z)$ and $\theta_3^N(z)$
are analytic in the interior of $V$. Then the function
$\mathrm{Re}\left(\theta_2^N-\theta_3^N\right)$ is a harmonic
function in the interior of $V$ and has constant value at the
boundary of $V$. Therefore, by the maximum modulus principle this
function must be a constant in $V$. This is not possible and hence
the curves $\mathfrak{H}_j^+$ cannot be connected to each other.

By inspecting the behavior of $\theta_2^N-\theta_3^N$ at $z=\infty$,
we see that one of these curves must be an open curve that
approaches infinity at a direction parallel to the imaginary axis.
We will call this curve $\mathfrak{H}_{\infty}^+$ and its reflection
with respect to the real axis $\mathfrak{H}_{\infty}^-$. Since the
other 2 curves cannot intersect each other, and they cannot go to
infinity either, they must end at the real axis and be connected to
the curves in $\mathfrak{H}^-$. We will call the curve on the left
hand side $\mathfrak{H}_L^+$ and the one on the right hand side
$\mathfrak{H}_R^+$. These two curves must end at different points on
the real axis as they cannot intersect. Let us denote the curves
$\mathfrak{H}_L$ and $\mathfrak{H}_R$ by
\begin{equation}\label{eq:HLR}
\begin{split}
\mathfrak{H}_L&=\mathfrak{H}^+_L\cup\mathfrak{H}^-_L\cup\{x_L\},\\
\mathfrak{H}_R&=\mathfrak{H}^+_R\cup\mathfrak{H}^-_R\cup\{x_R\}.
\end{split}
\end{equation}
where $\mathfrak{H}_L^{-}$ and $\mathfrak{H}_R^-$ are the
reflections of $\mathfrak{H}_L^{+}$ and $\mathfrak{H}_R^+$ with
respect to the real axis and $x_L$, $x_R$ are their accumulation
points on the real axis.
\begin{figure}
\centering \psfrag{l1}[][][1][0.0]{\small$\lambda_{1}^N$}
\psfrag{l+}[][][1][0.0]{\small$\lambda_{3}^N$}
\psfrag{l-}[][][1][0.0]{\small$\lambda_{4}^N$}
\psfrag{l2}[][][1][0.0]{\small$\lambda_{2}^N$}
\psfrag{h+inf}[][][1][0.0]{\small$\mathfrak{H}_{\infty}^+$}
\psfrag{h-inf}[][][1][0.0]{\small$\mathfrak{H}_{\infty}^-$}
\psfrag{hL}[][][1][0.0]{\small$\mathfrak{H}_{L}$}
\psfrag{hR}[][][1][0.0]{\small$\mathfrak{H}_{R}$}
\psfrag{C}[][][1][0.0]{\small$\mathcal{C}$}
\psfrag{iota}[][][1][0.0]{\small$x^{\ast}$}
\psfrag{o+}[][][1][0.0]{\small$\Omega_R$}
\psfrag{o-}[][][1][0.0]{\small$\Omega_L$}
\psfrag{o+1}[][][1][0.0]{\small$\Omega_2$}
\psfrag{o+2}[][][1][0.0]{\small$\Omega_1$}
\includegraphics[scale=0.8]{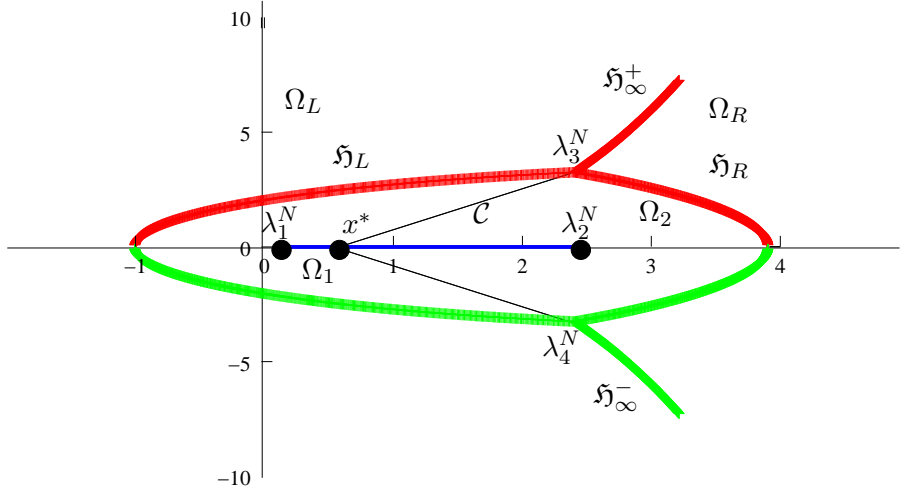}
\caption{The set $\mathfrak{H}$ for $a=0.9$, $\beta_N=0.7$ and
$c_N=0.4$. The branch points are given by $\lambda_1^N\approx
0.12518$, $\lambda_2^N\approx 2.48841$, $\lambda_{3}^N\approx
2.40520+3.2516i$ and $\lambda_4^N\approx 2.40520-3.2516i$. The point
$\iota$ in Lemma \ref{le:real23} is given by $\iota\approx0.602$.
For $a<1$, the function
$\mathrm{Re}\left(\theta_2^N(z)-\theta_3^N(z)\right)$ is negative in
the open region $\Omega_L$ on the left of $\mathfrak{H}_L$ and
positive in the region $\Omega_R$ on the right hand side of
$\mathfrak{H}_R$.}\label{fig:zeroset}
\end{figure}

Let us now show that
\begin{equation}\label{eq:H+}
\mathfrak{H}^+=\mathfrak{H}^+_{\infty}\cup\mathfrak{H}_L^+\cup\mathfrak{H}_R^+.
\end{equation}
Suppose there is a point $z_1\in\mathfrak{H}^+$ that does not belong
to any of the curves in the right hand side of (\ref{eq:H+}). Then
$z_1$ must belong to a curve $\mathfrak{H}_4^+\in\mathfrak{H}^+$. By
changing the definition of $\mathcal{C}$ again if necessary, we see
that $\mathfrak{H}_4^+$ must be smooth. This curve cannot end on the
real axis because by Lemma \ref{le:inter}, the set $\mathfrak{H}$
can at most intersect the real axis at 2 points and $\mathfrak{H}$
has already intersected the real axis at the 2 points $x_L$ and
$x_R$ in (\ref{eq:HLR}). The curve $\mathfrak{H}_4^+$ cannot
approach infinity or intersect any other curves in $\mathfrak{H}^+$
either and therefore it must be a close loop in the upper half
plane. As $\mathfrak{H}_4^+$ cannot intersect the curves in the
right hand side of (\ref{eq:H+}), the point $\lambda_3^N$ must lie
outside of the region $\tilde{V}$ bounded by $\mathfrak{H}_4^+$.
This would then imply that the harmonic function
$\mathrm{Re}\left(\theta_2^N-\theta_3^N\right)$ is constant inside
the region $\tilde{V}$, which is not possible and hence we have
\begin{equation*}
\mathfrak{H}=\mathfrak{H}^+_{\infty}\cup\mathfrak{H}^-_{\infty}\cup\mathfrak{H}_L\cup\mathfrak{H}_R.
\end{equation*}
Finally, if $(x_L,x_R)\cap[\lambda_1^N,\lambda_2^N]=\emptyset$, then
the function $\mathrm{Re}\left(\theta_2^N-\theta_3^N\right)$ is
harmonic inside the region bounded by $\mathfrak{H}_L$ and
$\mathfrak{H}_R$, which is not possible as it would imply that it is
a constant function in this region. This concludes the proof of the
proposition.
\end{proof}
The shape of the set $\mathfrak{H}$ is indicated in Figure
\ref{fig:zeroset}. The Octave generated figure shows the set for
$a=0.9$, $\beta_N=0.7$ and $c_N=0.4$.
\subsubsection{Jump discontinuities of the functions}
From now on, we will choose the branch cut $\mathcal{C}$ to be a
simple curve joining $\lambda_3^N$ and $\lambda_4^N$ that is
symmetric with respect to the real axis. We also require
$\mathcal{C}$ to lie between the curves $\mathfrak{H}_L$ and
$\mathfrak{H}_R$ in Proposition \ref{pro:shapeH} and that it
intersects $\mathbb{R}$ at a point
$\lambda_1^N<x^{\ast}<\lambda_2^N$. The integration contours for the
functions $\theta_j^N(z)$ in (\ref{eq:theta}) are chosen such that
they do not intersect the set $(-\infty,\lambda_2^N)\cup\mathcal{C}$
and the point $\lambda_l^N$ in (\ref{eq:theta}) is chosen to be
$\lambda_2^N$.

\begin{proposition}\label{pro:sheet}
If $a>1$, then $\lambda_2^N$ is a branch point of $\xi_3^N(z)$ and
$\lambda_1^N$ is a branch point of $\xi_2^N(z)$. On the other hand,
if $a<1$, then $\lambda_1^N$ is a branch point of $\xi_3^N(z)$ while
$\lambda_2^N$ is a branch point of $\xi_2^N(z)$.
\end{proposition}
\begin{proof} Let $l_1=2$, $l_2=3$ for $a>1$ and $l_1=3$, $l_2=2$ for $a<1$. Then
for large enough $z_0>\lambda_2^N$, all the functions $\xi_j^N$ are
real and we have $\xi_1^N(z_0)>\xi_{l_2}^N(z_0)>\xi_{l_1}^N(z_0)$ by
(\ref{eq:xiinfty}). This ordering must preserve at $\lambda_2^N$ as
the roots cannot coincide between $z_0$ and $\lambda_2^N$. At
$\lambda_2^N$, one of the roots must coincide with $\xi_1^N$ and
from the ordering $\xi_1^N(z_0)>\xi_{l_2}^N(z_0)>\xi_{l_1}^N(z_0)$,
we must have
$\xi_1^N(\lambda_2^N)=\xi_{l_2}^N(\lambda_2^N)=\gamma_2^N$. On the
other hand, for small enough $0<\epsilon<\lambda_1^N$, all three
$\xi_j^N$ will be real. From the asymptotic behavior of $\xi_1^N(z)$
at $0$ (\ref{eq:xizero}), we see that for small enough $\epsilon$,
$\xi_1^N(\epsilon)$ will be smaller than both $\xi_2^N(\epsilon)$
and $\xi_3^N(\epsilon)$. From the asymptotic behavior of $\xi_2^N$
and $\xi_3^N$ at $z=-\infty$ (\ref{eq:xiinfty}) and the fact that
these 2 functions has no singularity and cannot coincide in
$(-\infty, \lambda_1^N)$, we see that at $z=\epsilon$, we must have
$\xi_{l_2}^N(\epsilon)>\xi_{l_1}^N(\epsilon)$. Therefore we have
$\xi_{l_2}^N(\epsilon)>\xi_{l_1}^N(\epsilon)>\xi_1^N(\epsilon)$.
This ordering must again be preserved at $\lambda_1^N$. Therefore we
must have
$\xi_1^N(\lambda_1^N)=\xi_{l_1}^N(\lambda_1^N)=\gamma_1^N$. This
shows that $\lambda_2^N$ is a branch point of $\xi_{l_2}^N$ while
$\lambda_1^N$ is a branch point of $\xi_{l_1}^N$.
\end{proof}
We can now determine the jump discontinuities of $\xi_j^N$ on the
branch cuts
\begin{equation}\label{eq:bound}
\begin{split}
\xi_{1,\pm}^N(z)&=
               \xi_{2,\mp}^N(z),  \quad z\in\mathfrak{B}_{k_2},
\quad \xi_{1,\pm}^N(z)=
               \xi_{3,\mp}^N(z),  \quad z\in\mathfrak{B}_{k_3}, \\
\xi_{2,\pm}^N(z)&=\xi_{3,\mp}^N(z),\quad z\in\mathcal{C}.
\end{split}
\end{equation}
where $k_2=1$, $k_3=2$ for $a>1$ and $k_2=2$, $k_3=1$ for $a<1$ and
$\mathfrak{B}_j$ are defined by
\begin{equation}\label{eq:frakB}
\mathfrak{B}_1=[\lambda_1^N,x^{\ast}),\quad
\mathfrak{B}_2=(x^{\ast},\lambda_2^N].
\end{equation}
The branch cut structure of the Riemann surface $\Lie_N$ is
indicated in Figure \ref{fig:sheets2}.
\begin{figure}
\centering \psfrag{x1}[][][1][0.0]{\small$\xi_1^N$}
\psfrag{x2}[][][1][0.0]{\small$\xi_2^N$}
\psfrag{x3}[][][1][0.0]{\small$\xi_3^N$}
\includegraphics[scale=0.75]{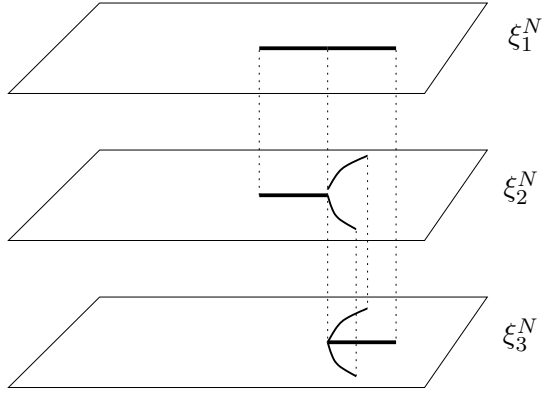} \caption{The branch cut
structure of the Riemann surface $\Lie_N$ when
$a>1$.}\label{fig:sheets2}
\end{figure}

Let us define $\tilde{\theta}_j^N(z)$ to be constant shifts of the
$\theta_j^N(z)$.
\begin{equation}\label{eq:thetatilde}
\begin{split}
\tilde{\theta}_1^N(z)&=\theta_1^N(z);\\
\tilde{\theta}_2^N(z)&=\left\{\begin{array}{ll}
             \theta_2^N(z)-\theta_{2,+}^N(\lambda_1^N)+\theta_{1,-}^N(\lambda^N_{1}), & \hbox{$a>1$;} \\
             \theta_2^N(z)-\theta_{2,+}^N(\lambda_2^N), & \hbox{$a<1$.}
           \end{array}
         \right.\\
\tilde{\theta}_3^N(z)&=\left\{
            \begin{array}{ll}
              \theta_3^N(z)-\theta_{3,+}^N(\lambda^N_{2}), & \hbox{$a>1$;} \\
              \theta_3^N(z)-\theta_{3,+}^N(\lambda_1^N)+\theta_{1,-}^N(\lambda^N_{1}), & \hbox{$a<1$.}
            \end{array}
          \right.
\end{split}
\end{equation}
Note that the difference between $\theta_2^N(z)$ and $\theta_3^N(z)$
are the same as the difference between $\tilde{\theta}_2^N(z)$ and
$\tilde{\theta}_3^N(z)$.
\begin{lemma}\label{le:diff23}
Let $\tilde{\theta}_j^N(z)$ be defined as in (\ref{eq:thetatilde}),
then we have
$\theta_2^N(z)-\theta_3^N(z)=\tilde{\theta}_2^N(z)-\tilde{\theta}_3^N(z)$.
If $z\in(-\infty,\lambda_2^N)\cup\mathcal{C}$, then we also have
\begin{equation}\label{eq:diff23}
\begin{split}
\theta_{2,\pm}^N(z)-\theta_{3,\pm}^N(z)&-\left(\tilde{\theta}_{2,\pm}^N(z)-\tilde{\theta}_{3,\pm}^N(z)\right)=0,\\
\theta_{2,\pm}^N(z)-\theta_{3,\mp}^N(z)&-\left(\tilde{\theta}_{2,\pm}^N(z)-\tilde{\theta}_{3,\mp}^N(z)\right)=0
\end{split}
\end{equation}
where the `$+$' and `$-$' subscripts indicates the boundary values
on the left and right hand sides of
$(-\infty,\lambda_2^N)\cup\mathcal{C}$.
\end{lemma}
\begin{proof}Let $l_1=2$, $l_2=3$ when $a>1$ and $l_1=3$, $l_2=2$ when $a<1$. Then from (\ref{eq:thetatilde}), we have
\begin{equation}\label{eq:diffcon}
\tilde{\theta}_2^N(z)-\tilde{\theta}_3^N(z)-\left(\theta_2^N(z)-\theta_3^N(z)\right)=(-1)^{l_1+1}\left(\theta_{l_1,+}^N(\lambda_{1}^N)
-\theta_{l_2,+}^N(\lambda_{2}^N)-\theta_{1,-}^N(\lambda_1^N)\right).
\end{equation}
From (\ref{eq:thetatilde}), it is clear that if
$z\in(-\infty,\lambda_2^N)\cup\mathcal{C}$, then the differences in
(\ref{eq:diff23}) are also given by the constant on the right hand
side of (\ref{eq:diffcon}).

Let us consider the difference $\theta_{l_1,+}^N(\lambda_{1}^N)
-\theta_{l_2,+}^N(\lambda_{2}^N)$. We will compute
$\theta_{l_1,+}^N(\lambda_{1}^N)$
($\theta_{l_2,+}^N(\lambda_{2}^N)$) with an integration path that
consists of 2 parts. The first part goes along the left (right) hand
side of $\mathcal{C}$ from $\lambda_3^N$ to $x^{\ast}$. The second
part goes along the positive side of the real axis from $x^{\ast}$
to $\lambda_1^N$ ($\lambda_2^N$). We then have
\begin{equation}\label{eq:theconst}
\begin{split}
\theta_{l_1,+}^N(\lambda_1^N)&=\int_{\lambda_3^N}^{x^{\ast}}\xi_{l_1,+}^N(x)dx+\int_{x^{\ast}}^{\lambda_1^N}\xi_{l_1,+}^N(x)dx,\\
\theta_{l_2,+}^N(\lambda_2^N)&=\int_{\lambda_3^N}^{x^{\ast}}\xi_{l_2,-}^N(x)dx+\int_{x^{\ast}}^{\lambda_2^N}\xi_{l_2,+}^N(x)dx
\end{split}
\end{equation}
From (\ref{eq:bound}), we see that $\xi_{l_1,+}^N=\xi_{l_2,-}^N$
along $\mathcal{C}$ and $\xi_{l_1,+}^N(x)=\xi_{1,-}^N(x)$ along
$(x^{\ast},\lambda_1^N)$, while $\xi_{l_2,+}^N(x)=\xi_{1,-}^N(x)$
along $(x^{\ast},\lambda_2^N)$. From this and (\ref{eq:theconst}),
we obtain
\begin{equation}\label{eq:diff1}
\begin{split}
\theta_{l_1,+}^N(\lambda_1^N)-\theta_{l_2,+}^N(\lambda_2^N)=\int_{\lambda_2^N}^{\lambda_1^N}\xi_{1,-}^N(x)dx=\theta_{1,-}^N(\lambda_1^N),
\end{split}
\end{equation}
where the last inequality follows from the fact that $\xi_1^N(z)$ is
analytic across $\mathcal{C}$ and hence we can choose the
integration path for $\theta_1^N(\lambda_1^N)$ to be along the real
axis.

From (\ref{eq:diff1}), we see that $\theta_{l_1,+}^N(\lambda_{1}^N)
-\theta_{l_2,+}^N(\lambda_{2}^N)-\theta_{1,-}^N(\lambda_1^N)=0$ and
hence
$\theta_2^N(z)-\theta_3^N(z)=\tilde{\theta}_2^N(z)-\tilde{\theta}_3^N(z)$.
\end{proof}
From the behavior of $\xi_j^N(z)$ on the cuts, we have the following
analyticity properties of the $\tilde{\theta}_j^N(z)$.
\begin{lemma}\label{le:cuttheta}
The function $\tilde{\theta}_1^N(z)$ is analytic on
$\mathbb{C}\setminus(-\infty,\lambda_2^N]$. The function
$\tilde{\theta}_2^N(z)$ ($\tilde{\theta}_3^N(z)$) is analytic on
$\mathbb{C}\setminus\left((-\infty,x^{\ast}]\cup\mathcal{C}\right)$
when $a>1$ ($a<1$) and it is analytic on
$\mathbb{C}\setminus\left((-\infty,\lambda_2^N]\cup\mathcal{C}\right)$
when $a<1$ ($a>1$). Let $k_2=1$, $k_3=2$ for $a>1$ and $k_2=2$,
$k_3=1$ for $a<1$. Let $\mathcal{C}_{\pm}$ are the intersections of
$\mathcal{C}$ with the upper/lower half planes. Then the integrals
$\tilde{\theta}_j^N(z)$ have the following jump discontinuities.
\begin{equation}\label{eq:cuttheta}
\begin{split}
\tilde{\theta}_{1,\pm}^N(z)&=\tilde{\theta}_{j,\mp}^N(z)+\upsilon_{\pm},
,\quad z\in \mathfrak{B}_{k_j},\quad j=2,3,\\
\tilde{\theta}_{2,\pm}^N(z)&=\tilde{\theta}_{3,\mp}^N(z),\quad z\in\mathcal{C}_+,\\
\tilde{\theta}_{2,\pm}^N(z)&=\tilde{\theta}_{3,\mp}^N(z)+2c_N\beta_N\pi i,\quad a>1,\quad z\in\mathcal{C}_-,\\
\tilde{\theta}_{2,\pm}^N(z)&=\tilde{\theta}_{3,\mp}^N(z)-2c_N(1-\beta_N)\pi i,\quad a<1,\quad z\in\mathcal{C}_-,\\
\tilde{\theta}_{1,+}^N(z)&=\tilde{\theta}_{1,-}^N(z)-2c_N\pi i,\quad
z\in (0,\lambda_1^N],\\
\tilde{\theta}_{1,+}^N(z)&=\tilde{\theta}_{1,-}^N(z)-2\pi i,\quad
z\in (-\infty,0),\\
\tilde{\theta}_{2,+}^N(z)&=\tilde{\theta}_{2,-}^N(z)+2c_N(1-\beta_N)\pi
i,\quad z\in
(-\infty,\lambda_{1}^N],\\
\tilde{\theta}_{2,+}^N(z)&=\tilde{\theta}_{2,-}^N(z)+2c_N(1-\beta_N)\pi
i,\quad z\in \mathfrak{B}_1,\quad a<1,\\
\tilde{\theta}_{3,+}^N(z)&=\tilde{\theta}_{3,-}^N(z)+2c_N\beta_N\pi
i,\quad z\in (-\infty,\lambda_{1}^N],\\
\tilde{\theta}_{3,+}^N(z)&=\tilde{\theta}_{3,-}^N(z)+2c_N\beta_N\pi
i,\quad z\in \mathfrak{B}_1,\quad a>1.
\end{split}
\end{equation}
where $\upsilon_{\pm}$ is the constant
\begin{equation}\label{eq:up}
\upsilon_{\pm}=\theta_{j,+}\left(\lambda_{k_j}^N\right)
-\theta_{j,\mp}\left(\lambda_{k_j}^N\right)+\theta_{1,\pm}\left(\lambda_{k_j}^N\right)-
\theta_{1,-}\left(\lambda_{k_j}^N\right).
\end{equation}
and $\mathfrak{B}_j$ is defined in (\ref{eq:frakB}). In particular,
by the jump discontinuities of the $\tilde{\theta}_j^N(z)$ at the
points $\lambda_k^N$, and the fact that $\theta_j^N(z)$ and
$\tilde{\theta}_j^N(z)$ differs by a constant shift only, we see
that the constant $\upsilon_{\pm}$ is either $-2c_N\beta_N\pi i$,
$-2c_N\left(1-\beta_N\right)\pi i$ or 0.
\end{lemma}
\begin{proof} From the jump discontinuities of $\xi_j^N(z)$ in
(\ref{eq:bound}), we see that
\begin{equation*}
\begin{split}
\left(\int_{\lambda_{k_j}^N}^z\xi_{1}^N(x)dx\right)_{\pm}&=\left(\int_{\lambda_{k_j}^N}^z\xi_{j}^N(x)dx\right)_{\mp},\quad z\in \mathfrak{B}_{k_j},\quad j=2,3\\
\left(\int_{\lambda_{3}^N}^z\xi_{2}^N(x)dx\right)_{\pm}&=\left(\int_{\lambda_{3}^N}^z\xi_{3}^N(x)dx\right)_{\mp},\quad
z\in\mathcal{C}_+,\\
\left(\int_{\lambda_{4}^N}^z\xi_{2}^N(x)dx\right)_{\pm}&=\left(\int_{\lambda_{4}^N}^z\xi_{3}^N(x)dx\right)_{\mp},\quad
z\in \mathcal{C}_-.
\end{split}
\end{equation*}
The path of integration in the first equation is taken along the
real axis and the integration paths in the last 2 equations are
taken along $\mathcal{C}$. By comparing this with (\ref{eq:theta})
and (\ref{eq:thetatilde}) and making use of Lemma \ref{le:diff23},
we obtain the first two equations in (\ref{eq:cuttheta}), together
with
\begin{equation}\label{eq:cutc-}
\tilde{\theta}_{2,\pm}^N(z)=\tilde{\theta}_{3,\mp}^N(z)+\theta_2^N(\lambda_4^N)-\theta_3^N(\lambda_4^N),\quad
z\in\mathcal{C}_-.
\end{equation}
Let us now compute the constant
$\theta_2^N(\lambda_4^N)-\theta_3^N(\lambda_4^N)$. Let $l_1=2$,
$l_2=3$ when $a>1$ and $l_1=3$, $l_2=2$ when $a<1$. Then
$\xi_{l_1}^N$ will be analytic on $(x^{\ast},\lambda_2^N]$ and hence
we can compute $\theta_{l_1}^N(\lambda_4^N)$ using a contour that
goes along the right hand side of $\mathcal{C}$. That is, we have
\begin{equation}\label{eq:thetal1}
\theta_{l_1}^N(\lambda_4^N)=\int_{\lambda_3^N}^{\lambda_4^N}\xi_{l_1,-}^N(x)dx.
\end{equation}
where the integration is performed along $\mathcal{C}$. To compute
$\theta_{l_2}^N(\lambda_4^N)$, let us choose an integration contour
as follows. The integration contour consists of three parts. The
first part goes from $\lambda_3^N$ to $x^{\ast}$ on the left hand
side of $\mathcal{C}$. The second part is a closed loop
$\mathcal{S}$ that goes from $x^{\ast}$ to $x^{\ast}$ with the
branch cut $\mathfrak{B}_2\cup\mathcal{C}$ of $\xi_{l_2}^N$ inside
it. The last part goes from $x^{\ast}$ to $\lambda_4^N$ along the
left hand side of $\mathcal{C}$. That is, we have
\begin{equation*}
\theta_{l_2}^N(\lambda_4^N)=\int_{\lambda_3^N}^{x^{\ast}}\xi_{l_2,+}^N(x)dx+\oint_{\mathcal{S}}\xi_{l_2}^N(x)dx
+\int_{x^{\ast}}^{\lambda_4^N}\xi_{l_2,+}^N(x)dx.
\end{equation*}
From the jump discontinuities of $\xi_j^N(x)$ on $\mathcal{C}$, we
obtain
\begin{equation*}
\theta_{l_2}^N(\lambda_4^N)=\int_{\lambda_3^N}^{\lambda_4^N}\xi_{l_1,-}^N(x)dx+\oint_{\mathcal{S}}\xi_{l_2}^N(x)dx.
\end{equation*}
Since branch cuts of $\xi_{l_2}^N$ are inside the loop
$\mathcal{S}$, this loop can be deformed into a loop around
$z=\infty$. By computing the integral using residue theorem, we
obtain
\begin{equation*}
\begin{split}
\theta_{2}^N(\lambda_4^N)&-\theta_{3}^N(\lambda_4^N)=-2\pi
ic_N(1-\beta_N),\quad a<1,\\
\theta_{2}^N(\lambda_4^N)&-\theta_{3}^N(\lambda_4^N)=2\pi
ic_N\beta_N,\quad a>1.
\end{split}
\end{equation*}
This, together with (\ref{eq:cutc-}) gives the third and fourth
equations in (\ref{eq:cuttheta}).

Let us now show that
\begin{equation}\label{eq:theta2}
\tilde{\theta}_{2,+}^N(z)=\tilde{\theta}_{2,-}^N(z)+2c_N(1-\beta_N)\pi
i,\quad z\in (-\infty,\lambda_{1}^N].
\end{equation}
The corresponding equation for $\tilde{\theta}_1^N(z)$ and
$\tilde{\theta}_3^N(z)$ can be proven in a similar way.
\begin{figure}
\centering \psfrag{l1}[][][1][0.0]{\small$\lambda_1^N$}
\psfrag{l2}[][][1][0.0]{\small$\lambda_2^N$}
\psfrag{l3}[][][1][0.0]{\small$\lambda_3^N$}
\psfrag{l4}[][][1][0.0]{\small$\lambda_4^N$}
\psfrag{g+}[][][1][0.0]{\small$\Gamma_+$}
\psfrag{g-}[][][1][0.0]{\small$\Gamma_-$}
\psfrag{z}[][][1][0.0]{\small$z$}
\includegraphics[scale=0.75]{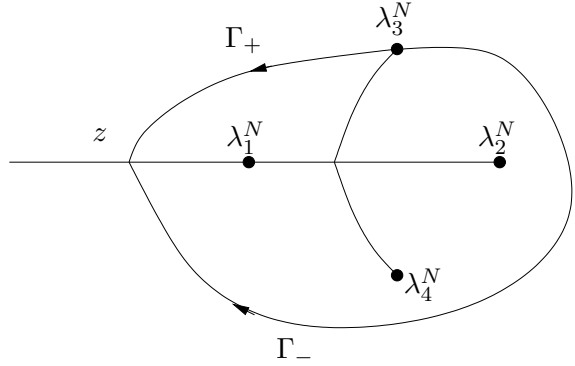}
\caption{The contours $\Gamma_+$ and $\Gamma_-$.}\label{fig:Gammapm}
\end{figure}
Let $z\in(-\infty,\lambda_{1}^N]$ and let $\Gamma_{\pm}$ be contours
from $\lambda_{3}^N$ to $z$ indicated as in Figure
\ref{fig:Gammapm}. Then we have
\begin{equation*}
\theta_{2,\pm}^N(z)=\int_{\Gamma_{\pm}}\xi_2^N(x)dx.
\end{equation*}
Let $\Gamma$ be the close loop on $\mathbb{C}$ such that
$\Gamma=\Gamma_+-\Gamma_-$, then we have
\begin{equation*}
\theta_{2,+}^N(z)=\theta_{2,-}^N(z)+\oint_{\Gamma}\xi_2^N(x)dx,\quad
z\in (-\infty,\lambda_{1}^N].
\end{equation*}
By Cauchy's theorem, we can deform the loop $\Gamma$ such that
$\Gamma$ becomes a loop around $z=\infty$. By computing the residue,
we obtain
\begin{equation}\label{eq:theta20}
\theta_{2,+}^N(z)=\theta_{2,-}^N(z)+2c_N(1-\beta_N)\pi i,\quad z\in
(-\infty,\lambda_{1}^N].
\end{equation}
Since $\tilde{\theta}_2^N(z)$ is a constant shift of
$\theta_2^N(z)$, (\ref{eq:theta20}) implies (\ref{eq:theta2}).

By using similar argument, we can obtain the rest of the jump
discontinuities.
\end{proof}
We will conclude this section with the following results on the
relative sizes of the
$\mathrm{Re}\left(\tilde{\theta}_j^N(z)\right)$, which are essential
in the implementation of the Riemann-Hilbert method.
\begin{lemma}\label{le:size}
The real parts of $\tilde{\theta}_j^N(z)$ in (\ref{eq:thetatilde})
are continuous in $\mathbb{R}_+\setminus[\lambda_{1},\lambda_{2}]$
and we have the followings.
\begin{enumerate}
\item Let $x_L<x_R$ be the points where $\mathfrak{H}$ intersects
the real axis, then we have
\begin{equation}\label{eq:size}
\begin{split}
&\mathrm{Re}\left(\tilde{\theta}_1^N(z)-\tilde{\theta}_{j}^N(z)\right)>0,\quad
z\in\mathbb{R}_+\setminus[x_L,x_R] ,\quad j=1,2,\\
&\mathrm{Re}\left(\tilde{\theta}_1^N(z)-\tilde{\theta}_{l_1}^N(z)\right)>0,\quad
z\in(0,\lambda_1^N),\\
&\mathrm{Re}\left(\tilde{\theta}_1^N(z)-\tilde{\theta}_{l_2}^N(z)\right)>0,\quad
z\in(\lambda_2^N,\infty).
\end{split}
\end{equation}
where $l_1=2$, $l_2=3$ for $a>1$ and $l_1=3$, $l_2=2$ for $a<1$.
\item Let $x^{\ast}$ be the intersection point between $\mathcal{C}$ and $\mathbb{R}$.
Then in a neighborhood $D_{x^{\ast}}$ of $x^{\ast}$, we have
\begin{equation}\label{eq:sizex}
\begin{split}
\mathrm{Re}\left(\tilde{\theta}_1^N(z)-\tilde{\theta}_{l_1,-}^N(z)\right)<0,\quad
z\in D_{x^{\ast}}\cap\mathcal{C}.
\end{split}
\end{equation}
where the `$+$' and `$-$' subscripts denote the boundary values at
the left and right hand sides of $\mathcal{C}$.
\item On $\mathcal{C}$, we have
\begin{equation}\label{eq:sizeC}
\begin{split}
\mathrm{Re}\left(\tilde{\theta}_{l_1,+}^N(z)-\tilde{\theta}_{l_1,-}^N(z)\right)<0,\quad
z\in \mathcal{C}.
\end{split}
\end{equation}
where the `$+$' and `$-$' subscripts denote the boundary values at
the left and right hand sides of $\mathcal{C}$.
\end{enumerate}
\end{lemma}
\begin{proof} From (\ref{eq:cuttheta}), we see that the real parts
of $\tilde{\theta}_j^N(z)$ are continuous in
$\mathbb{R}_+\setminus[\lambda_{1}^N,\lambda_{2}^N]$. Now from the
proof of Proposition \ref{pro:sheet}, we have, for $j=2,3$,
\begin{equation}\label{eq:ineq0}
\begin{split}
\mathrm{Re}\left(\xi_1^N(z)-\xi_j^N(z)\right)&>0,\quad
z\in(\lambda_{2}^N,\infty),\\
\mathrm{Re}\left(\xi_1^N(z)-\xi_j^N(z)\right)&<0,\quad
z\in(0,\lambda_{1}^N).
\end{split}
\end{equation}
From the definitions of $\theta_j^N(z)$ (\ref{eq:theta}) and
$\tilde{\theta}_j^N(z)$ (\ref{eq:thetatilde}), we obtain the second
and the third equations in (\ref{eq:size}). To prove the first
equation in (\ref{eq:size}), note that if $z\in(x_R,\infty)$, then
$z$ is in the region $\Omega_R$ in Figure \ref{fig:zeroset}.
Similarly, if $z\in(0,x_L)$, then $z\in\Omega_L$. By considering the
behavior of $\theta_2^N(z)$ and $\theta_3^N(z)$ at $z=\infty$, we
see that
\begin{equation}\label{eq:outside}
\begin{split}
&\mathrm{Re}\left(\theta_{l_2}^N(z)-\theta_{l_1}^N(z)\right)>0,\quad
z\in(x_R,\infty),\\
&\mathrm{Re}\left(\theta_{l_1}^N(z)-\theta_{l_2}^N(z)\right)>0,\quad
z\in(0,x_L).
\end{split}
\end{equation}
From Proposition \ref{pro:shapeH}, we have
$(x_L,x_R)\cap[\lambda_1^N,\lambda_2^N]\neq\emptyset$. This implies
$x_L<x^{\ast}<\lambda_2^N$ and $x_R>x^{\ast}>\lambda_1^N$.
Therefore, by the second, third equations in (\ref{eq:size}), Lemma
\ref{le:diff23}, (\ref{eq:outside}) and the fact that
$\xi_1^N(z)-\xi_{l_j}^N(z)$ is purely imaginary on
$\mathfrak{B}_{j}$, we obtain
\begin{equation*}
\begin{split}
&\mathrm{Re}\left(\tilde{\theta}_{1}^N(z)-\tilde{\theta}_{l_1}^N(z)\right)>
\mathrm{Re}\left(\tilde{\theta}_{1}^N(z)-\tilde{\theta}_{l_2}^N(z)\right)\geq
0,\quad
z\in(x_R,\infty),\\
&\mathrm{Re}\left(\tilde{\theta}_{1}^N(z)-\tilde{\theta}_{l_2}^N(z)\right)>
\mathrm{Re}\left(\tilde{\theta}_{1}^N(z)-\tilde{\theta}_{l_1}^N(z)\right)\geq
0,\quad z\in(0,x_L).
\end{split}
\end{equation*}
This, together with the second and the third equations in
(\ref{eq:size}), give the first equation in (\ref{eq:size}).

We will now prove (\ref{eq:sizex}). Let us consider the left hand
sides of (\ref{eq:sizex}) at the point $x^{\ast}$. Since
$\xi_1^N(z)-\xi_{l_j}^N(z)$ is purely imaginary on
$\mathfrak{B}_{j}$, we have
\begin{equation}\label{eq:ima}
\begin{split}
\int_{x^{\ast}}^{\lambda_{j}^N}\mathrm{Re}\left(\xi_1^N(x)-\xi_{l_j}^N(x)\right)dx=0.
\end{split}
\end{equation}
where the integration is performed along the real axis. Therefore we
have
\begin{equation}\label{eq:realx0}
\begin{split}
\mathrm{Re}\left(\tilde{\theta}_{l_1,+}^N(x^{\ast})\right)=\mathrm{Re}\left(\tilde{\theta}_{l_2,-}^N(x^{\ast})\right)
=\mathrm{Re}\left(\tilde{\theta}_1^N(x^{\ast})\right)
\end{split}
\end{equation}
where the `$+$' and `$-$' subscripts indicate the boundary values on
the left and right hand sides of $\mathcal{C}$. Now note that
$\tilde{\theta}_{l_1,+}^N(x^{\ast})$ is evaluated in the region
$\Omega_1$ of Figure \ref{fig:zeroset} and
$\tilde{\theta}_{l_2,-}^N(x^{\ast})$ is evaluated in the region
$\Omega_2$. Again, by considering the behavior of
$\tilde{\theta}_2^N(z)$ and $\tilde{\theta}_3^N(z)$ at $z=\infty$
and using the fact that
$\mathrm{Re}\left(\tilde{\theta}_2^N-\tilde{\theta}_3^N\right)$ can
only change signs across the sets $\mathfrak{H}$ and $\mathcal{C}$,
we obtain
\begin{equation}\label{eq:outside1}
\begin{split}
\pm\mathrm{Re}\left(\tilde{\theta}_{l_2,\pm}^N(x^{\ast})-\tilde{\theta}_{l_1,\pm}^N(x^{\ast})\right)>0.
\end{split}
\end{equation}
Then from (\ref{eq:realx0}) and (\ref{eq:outside1}), we obtain
(\ref{eq:sizex}) at $z=x^{\ast}$. The statement for $z\in
D_{x^{\ast}}\cap\mathcal{C}$ now follows from the continuities of
the functions in (\ref{eq:sizex}) along $\mathcal{C}$.

Finally, note that inside the set $\Omega_2$ between $\mathcal{C}$
and $\Xi_R$, (See Figure \ref{fig:zeroset}) we have the following
inequalities.
\begin{equation}\label{eq:O2}
\begin{split}
\mathrm{Re}\left(\tilde{\theta}_{l_1}^N(z)-\tilde{\theta}_{l_2}^N(z)\right)>0,\quad
z\in \Omega_2,
\end{split}
\end{equation}
by considering $z\in\Omega_2$ on the right hand side of
$\mathcal{C}$ in (\ref{eq:O2}) and making use of the jump conditions
between $\tilde{\theta}_2^N$ and $\tilde{\theta}_3^N$ in
(\ref{eq:cuttheta}), we obtain (\ref{eq:sizeC}).
\end{proof}
The final result in this section deals with the behavior of these
real parts in a neighborhood of the interval
$[\lambda_1^N,\lambda_2^N]$.
\begin{lemma}\label{le:lens}
The open intervals $(\lambda_{1}^N,x^{\ast})$ and
$(x^{\ast},\lambda_2^N)$ each has a neighborhood $U_1$ and $U_2$ in
the complex plane such that
\begin{equation}\label{eq:lens}
\mathrm{Re}\left(\tilde{\theta}_{j}^N(z)-\tilde{\theta}_1^N(z)\right)>0,\quad
z\in U_{k_j},\quad j=2,3.
\end{equation}
where $k_2=1$, $k_3=2$ for $a>1$ and $k_2=2$, $k_3=1$ for $a<1$ and
$\tilde{\theta}_j^N(z)$ are defined in (\ref{eq:thetatilde}).
\end{lemma}
\begin{proof}Since $\xi_1^N(z)-\xi_{j}^N(z)$ is purely imaginary in
$\mathfrak{B}_{k_j}$, we have
\begin{equation}\label{eq:boundvalue}
\begin{split}
\int_{\lambda_{k_j}^N}^{z}\mathrm{Re}\left(\xi_1^N(x)-\xi_{j}(x)\right)dx=0,\quad
z\in \mathfrak{B}_{k_j},\quad j=1,2,
\end{split}
\end{equation}
where $\mathfrak{B}_{k_j}$ is defined in (\ref{eq:frakB}).

Let
$\mathfrak{B}_{k_j}^0=\mathfrak{B}_{k_j}\setminus\{\lambda_{k_j}^N\}$.
Then on the positive and negative sides of $\mathfrak{B}_{k_j}^0$,
the derivatives of the functions
$\tilde{\theta}_{1,\pm}^N(z)-\tilde{\theta}_{j,\pm}^N(z)$ are given
by $\xi_{1,\pm}^N(z)-\xi_{j,\pm}^N(z)$ and are purely imaginary. In
fact, since $\xi_{1,+}(z)=m_{\hat{F}_N}$, we see that
$\xi_{1,+}^N(z)-\xi_{j,+}^N(z)=2\pi i\rho_N(z)$ where $\rho_N(z)>0$
is the density function of $\hat{F}_N$. On the other hand, by the
jump discontinuities (\ref{eq:bound}), we see that
$\xi_{1,-}^N(z)-\xi_{j,-}^N(z)=-2\pi i\rho_N(z)$. Hence by the
Cauchy Riemann equation, the real part of
$\tilde{\theta}_1^N(z)-\tilde{\theta}_{j}^N(z)$ is decreasing as we
move from $\mathfrak{B}_{k_j}^0$ into the upper half plane. From
(\ref{eq:boundvalue}) and (\ref{eq:thetatilde}), we see that
$\mathrm{Re}\left(\tilde{\theta}_1^N(z)-\tilde{\theta}_{j}^N(z)\right)<0$
for $z$ in the upper half plane near $\mathfrak{B}_{k_j}^0$.
Similarly, we also have
$\mathrm{Re}\left(\tilde{\theta}_1^N(z)-\tilde{\theta}_{j}^N(z)\right)<0$
for $z$ in the lower half plane near $\mathfrak{B}_{k_j}^0$. This
implies (\ref{eq:lens}) is true in a neighborhood $U_{k_j}$ of
$\mathfrak{B}_{k_j}^0$.
\end{proof}
\section{Riemann-Hilbert analysis}\label{se:RHP}
We can now implement the Riemann-Hilbert method to obtain the strong
asymptotics for the multiple Laguerre polynomials introduced in
Section \ref{se:MOP} and use it to prove Theorem \ref{thm:main2}.
The analysis is very similar to those in \cite{BKext2} (See also
\cite{Lysov}).

Let $C(f)$ be the Cauchy transform of the function $f(z)\in
L^2(\mathbb{R}_+)$ in $\mathbb{R}_+$
\begin{equation}\label{eq:cauchy}
C(f)(z)=\frac{1}{2\pi i}\int_{\mathbb{R}_+}\frac{f(s)}{s-z}ds,
\end{equation}
and let $w_1(z)$ and $w_2(z)$ be the weights of the multiple
Laguerre polynomials.
\begin{equation}\label{eq:weight}
w_1(z)=z^{M-N}e^{-Mz},\quad w_2(z)=z^{M-N}e^{-Ma^{-1}z},
\end{equation}
Denote by $\kappa_1$ and $\kappa_2$ the constants
\begin{equation*}
\kappa_1=-2\pi
i\left(h^{(1)}_{N_0-1,N_1}\right)^{-1},\quad\kappa_2=-2\pi
i\left(h^{(2)}_{N_0,N_1-1}\right)^{-1}.
\end{equation*}
Then due to the orthogonality condition (\ref{eq:multiop}), the
following matrix
\begin{equation}\label{eq:Ymatr}
Y(z)=\begin{pmatrix}P_{N_0,N_1}(z)&C(P_{N_0,N_1}w_1)(z)&C(P_{N_0,N_1}w_2)(z)\\
\kappa_1P_{N_0-1,N_1}(z)&\kappa_1C(P_{N_0-1,N_1}w_1)(z)&\kappa_1C(P_{N_0-1,N_1}w_2)(z)\\
\kappa_2P_{N_0,N_1-1}(z)&\kappa_2C(P_{N_0,N_1-1}w_1)(z)&\kappa_2C(P_{N_0,N_1-1}w_2)(z)
\end{pmatrix}
\end{equation}
is the unique solution of the following Riemann-Hilbert problem.
\begin{equation}\label{eq:RHPY}
\begin{split}
1.\quad &\text{$Y(z)$ is analytic in
$\mathbb{C}\setminus\mathbb{R}_+$},\\
2.\quad &Y_+(z)=Y_-(z)\begin{pmatrix}1&w_1(z)&w_2(z)\\
0&1&0\\
0&0&1
\end{pmatrix},\quad z\in\mathbb{R}_+\\
3.\quad &Y(z)=\left(I+O(z^{-1})\right)\begin{pmatrix}z^{N}&0&0\\
0&z^{-N_0}&0\\
0&0&z^{-N_1}
\end{pmatrix},\quad z\rightarrow\infty,\\
4.\quad  &Y(z)=O(1),\quad z\rightarrow 0.
\end{split}
\end{equation}
By a similar computation as the one in \cite{BKMOP} and \cite{BK1},
we can express the kernel (\ref{eq:ker}) in terms of the solution of
the Riemann-Hilbert problem $Y(z)$.
\begin{equation}\label{eq:kerRHP}
\begin{split}
K_{M,N}(x,y)&=\frac{(xy)^{\frac{M-N}{2}}\left(e^{-My}\left[Y^{-1}_+(y)Y_+(x)\right]_{21}
+e^{-Ma^{-1}y}\left[Y^{-1}_+(y)Y_+(x)\right]_{31}\right)}{2\pi
i(x-y)},\\
&=\frac{(xy)^{\frac{M-N}{2}}}{2\pi i(x-y)}\left(0\quad e^{-My}\quad
e^{-Ma^{-1}y}\right)Y_+^{-1}(y)Y_+(x)\begin{pmatrix} 1 \\ 0\\0
\end{pmatrix}
\end{split}
\end{equation}
where $A_{21}$ and $A_{31}$ are the $21^{th}$ and $31^{th}$ entries
of $A$.
\subsection{First transformation of the Riemann-Hilbert problem}
We should now use the functions $\tilde{\theta}_j^N(z)$ in
(\ref{eq:thetatilde}) to deform the Riemann-Hilbert problem
(\ref{eq:RHPY}). Our goal is to deform the Riemann-Hilbert problem
so that it can be approximated by a Riemann-Hilbert problem that is
explicitly solvable. As in \cite{BKext2} and \cite{Lysov}, the set
$\mathfrak{H}$ in (\ref{eq:frakH}) will be important to the
construction.

We will now start deforming the Riemann-Hilbert problem
(\ref{eq:RHPY}). First let us define the functions $g_j^N(z)$ to be
\begin{equation}\label{eq:g}
\begin{split}
g_1^N(z)&=\tilde{\theta}_1^N(z)+(1-c_N)\log z,\quad
g_2^N(z)=\tilde{\theta}_2^N(z)+z,
\\
g_3^N(z)&=\tilde{\theta}_3^N(z)+\frac{z}{a}.
\end{split}
\end{equation}
where $\tilde{\theta}_j^N(z)$ is defined in (\ref{eq:thetatilde})
and the branch cut of $\log z$ in $g_1^N(z)$ is chosen to be the
negative real axis.

We then define $T(z)$ to be
\begin{equation}\label{eq:Tz}
T(z)=diag\left(e^{-Ml_1^N},e^{-M\tilde{l}_2^N},e^{-M\tilde{l}_3^N}\right)Y(z)
diag\left(e^{Mg_1^N(z)},e^{Mg_2^N(z)},e^{Mg_3^N(z)}\right),
\end{equation}
where $\tilde{l}_2^N$ and $\tilde{l}_3^N$ are given by
\begin{equation*}
\begin{split}
\tilde{l}_2^N&=\left\{
                \begin{array}{ll}
                  l_2^N-\theta_{2,+}^N(\lambda_1^N)+\theta_{1,-}^N(\lambda^N_{1}), & \hbox{$a>1$;} \\
                  l_2^N-\theta_{2,+}^N(\lambda_2^N), & \hbox{$a<1$.}
                \end{array}
              \right.\\
\tilde{l}_3^N&=\left\{\begin{array}{ll}
              l_3^N(z)-\theta_{3,+}^N(\lambda^N_{2}), & \hbox{$a>1$;} \\
              l_3^N(z)-\theta_{3,+}^N(\lambda_1^N)+\theta_{1,-}^N(\lambda^N_{1}), & \hbox{$a<1$.}
            \end{array}
          \right.
\end{split}
\end{equation*}

The matrix $T(z)$ will satisfy the following Riemann-Hilbert
problem.
\begin{equation}\label{eq:RHPT}
\begin{split}
1.\quad &\text{$T(z)$ is analytic in
$\mathbb{C}\setminus\left(\mathbb{R}\cup\mathcal{C}\right)$},\\
2.\quad &T_+(z)=T_-(z)J_T(z),\quad z\in\mathbb{R}\cup\mathcal{C},\\
3.\quad &T(z)=I+O(z^{-1}),\quad z\rightarrow\infty,\\
4.\quad  & T(z)=O(1),\quad z\rightarrow 0.
\end{split}
\end{equation}
where $J_T(z)$ is the following matrix
\begin{equation}\label{eq:JT}
\begin{split}
J_T(z)=\begin{pmatrix}e^{M\left(\tilde{\theta}_{1,+}^N(z)-\tilde{\theta}_{1,-}^N(z)\right)}&
e^{M\left(\tilde{\theta}_{2,+}^N(z)-\tilde{\theta}_{1,-}^N(z)\right)}&
e^{M\left(\tilde{\theta}_{3,+}^N(z)-\tilde{\theta}_{1,-}^N(z)\right)}\\
0&e^{M\left(\tilde{\theta}_{2,+}^N(z)-\tilde{\theta}_{2,-}^N(z)\right)}&0\\
0&0&e^{M\left(\tilde{\theta}_{3,+}^N(z)-\tilde{\theta}_{3,-}^N(z)\right)}
\end{pmatrix},
\end{split}
\end{equation}
By applying Lemma \ref{le:cuttheta} to the $\tilde{\theta}_j^N(z)$,
we can simplify the jump matrix $J_T(z)$. In particular, on
$\mathfrak{B}_{k_2}$ in (\ref{eq:frakB}), we have
\begin{equation}\label{eq:JTk2}
\begin{split}
J_T(z)&=\begin{pmatrix}e^{M\left(\tilde{\theta}_1^N(z)-\tilde{\theta}_2^N(z)\right)_+}&
1&
e^{M\left(\tilde{\theta}_{3,+}^N(z)-\tilde{\theta}_{1,-}^N(z)\right)}\\
0&e^{M\left(\tilde{\theta}_1^N(z)-\tilde{\theta}_2^N(z)\right)_-}&0\\
0&0&1
\end{pmatrix},
\end{split}
\end{equation}
while on $\mathfrak{B}_{k_3}$, we have
\begin{equation}\label{eq:JTk3}
\begin{split}
J_T(z)&=\begin{pmatrix}e^{M\left(\tilde{\theta}_1^N(z)-\tilde{\theta}_3^N(z)\right)_+}&
e^{M\left(\tilde{\theta}_{2,+}^N(z)-\tilde{\theta}_{1,-}^N(z)\right)}&
1\\
0&1&0\\
0&0&e^{M\left(\tilde{\theta}_1^N(z)-\tilde{\theta}_3^N(z)\right)_-}
\end{pmatrix},
\end{split}
\end{equation}
On the rest of the positive real axis, the jump matrix becomes
\begin{equation}\label{eq:JTrest}
\begin{split}
J_T(z)&=\begin{pmatrix}1&
e^{M\left(\tilde{\theta}_{2,+}^N(z)-\tilde{\theta}_{1,-}^N(z)\right)}&
e^{M\left(\tilde{\theta}_{3,+}^N(z)-\tilde{\theta}_{1,-}^N(z)\right)}\\
0&1&0\\
0&0&1
\end{pmatrix}.
\end{split}
\end{equation}
This is because $Mc_N\beta_N=N_1$, $Mc_N(1-\beta_N)=N_0$ and
$Mc_N=N$ are all integers. And on the negative real axis, the matrix
$T(z)$ has no jump for the same reason. Note that the jump matrix
$J_T(z)$ is continuous at $z=0$ as the off-diagonal entries of
(\ref{eq:JTrest}) contain the factor $e^{-M\theta_{1,-}^N(z)}$ which
vanishes at the origin.

The jump matrix on $\mathcal{C}$ is given by
\begin{equation}\label{eq:JTC}
\begin{split}
J_T(z)&=\begin{pmatrix}1& 0&0\\
0&e^{M\left(\tilde{\theta}_{2,+}^N(z)-\tilde{\theta}_{2,-}^N(z)\right)}&0\\
0&0&e^{M\left(\tilde{\theta}_{3,+}^N(z)-\tilde{\theta}_{3,-}^N(z)\right)}
\end{pmatrix}.
\end{split}
\end{equation}
The Riemann-Hilbert problem for $T(z)$ now takes same form as the
one in \cite{BKext2} (See also \cite{Lysov}) and the techniques
developed there can now be applied to our problem.

\subsection{Lens opening and approximation of the Riemann-Hilbert
problem}

We will now apply the global lens opening technique developed in
\cite{BKext2}. We will have to use the properties of the set
$\mathfrak{H}$ in (\ref{eq:frakH}) to define the global lens
contour.

By Proposition \ref{pro:shapeH}, the set union of the set
$\mathfrak{H}$ and $\mathcal{C}$ divides the complex plane into 4
regions $\Omega_L$, $\Omega_R$, $\Omega_1$ and $\Omega_2$. The
region $\Omega_L$ ($\Omega_R$) is an open region that lies on the
left (right) hand side of the contours $\mathfrak{H}_{\pm}^{\infty}$
and $\mathfrak{H}_L$. ($\mathfrak{H}_R$) The region $\Omega_1$
($\Omega_2$) is the region bounded by $\mathfrak{H}_L$
($\mathfrak{H}_R$) and $\mathcal{C}$. (See Figure \ref{fig:zeroset})

By considering the behavior of $\theta_2^N(z)-\theta_3^N(z)$ near
$z=\infty$ in (\ref{eq:asymtheta}), we see that for $a<1$ ($a>1$),
$\mathrm{Re}\left(\theta_2^N(z)-\theta_3^N(z)\right)$ is negative
(positive) in $\Omega_L$ and positive (negative) in $\Omega_R$. Let
us define the contour $\Xi_L$ ($\Xi_R$) to be a contour from
$\lambda_{4}^N$ to $\lambda_3^N$ in $\Omega_L$ ($\Omega_R$) such
that $\Xi_L$ intersects $\mathbb{R}$ at a point $x_1<0$ and $\Xi_R$
intersects $\mathbb{R}$ at a point $x_2>\lambda_2^N$. (See Figure
\ref{fig:sigma})
\begin{figure}
\centering \psfrag{l1}[][][1][0.0]{\small$\lambda_{1}^N$}
\psfrag{l2}[][][1][0.0]{\small$\lambda_{2}^N$}
\psfrag{l+}[][][1][0.0]{\small$\lambda_3^N$}
\psfrag{l-}[][][1][0.0]{\small$\lambda_4^N$}
\psfrag{iota}[][][1][0.0]{\small$x^{\ast}$}
\psfrag{sigmal}[][][1][0.0]{\small$\Xi_L$}
\psfrag{sigmar}[][][1][0.0]{\small$\Xi_R$}
\includegraphics[scale=0.65]{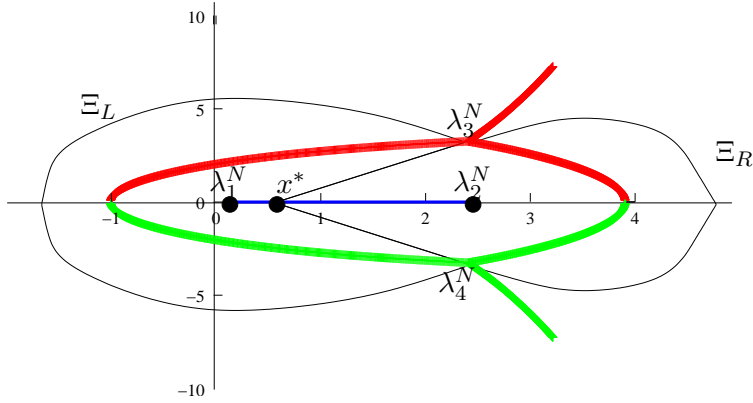}
\caption{The contours $\Xi_L$ and $\Xi_R$.}\label{fig:sigma}
\end{figure}

Let us define the lens contours $\Xi_{\pm}^{j}$, $j=1,2$ around the
branch cut $[\lambda^N_{1},\lambda^N_{2}]$ as follows. The contours
$\Xi_{\pm}^j$ are contours in the neighborhoods $U_j$ in Lemma
\ref{le:lens} joining $\lambda_j^N$ and $\mathcal{C}$ in the
upper/lower half plane. Together with the contours $\Xi_L$ and
$\Xi_R$, the lens contours are depicted in Figure \ref{fig:lens2}.
\begin{figure}
\centering \psfrag{l1}[][][1][0.0]{\small$\lambda_{1}^N$}
\psfrag{l2}[][][1][0.0]{\small$\lambda_{2}^N$}
\psfrag{l+}[][][1][0.0]{\small$\lambda_3^N$}
\psfrag{l-}[][][1][0.0]{\small$\lambda_4^N$}
\psfrag{xi1+}[][][1][0.0]{\small$\Xi_+^1$}
\psfrag{xi1-}[][][1][0.0]{\small$\Xi_-^1$}
\psfrag{xi2+}[][][1][0.0]{\small$\Xi_+^2$}
\psfrag{xi2-}[][][1][0.0]{\small$\Xi_-^2$}
\psfrag{xil}[][][1][0.0]{\small$\Xi_L$}
\psfrag{xir}[][][1][0.0]{\small$\Xi_R$}
\psfrag{x1}[][][1][0.0]{\small$x_1$}
\psfrag{x2}[][][1][0.0]{\small$x_2$}
\psfrag{x0}[][][1][0.0]{\small$x^{\ast}$}
\psfrag{L1+}[][][1][0.0]{\small$L_+^1$}
\psfrag{L2+}[][][1][0.0]{\small$L_+^2$}
\psfrag{L1-}[][][1][0.0]{\small$L_-^1$}
\psfrag{L2-}[][][1][0.0]{\small$L_-^2$}
\psfrag{ol}[][][1][0.0]{\small$\mathbb{D}_1$}
\psfrag{or}[][][1][0.0]{\small$\mathbb{D}_2$}
\includegraphics[scale=0.75]{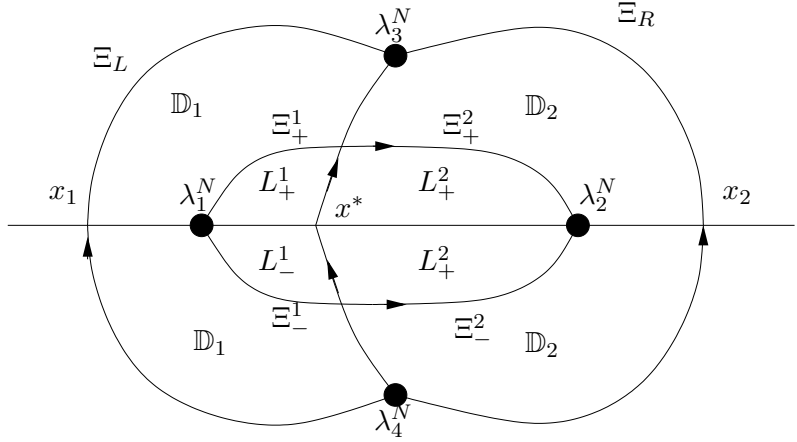}
\caption{The lens contours.}\label{fig:lens2}
\end{figure}
Let the matrices $\mathcal{G}_1$, $\mathcal{G}_2$, $\mathcal{G}_3$,
$\mathcal{G}_4$ be the followings
\begin{equation}\label{eq:Kj}
\begin{split}
\mathcal{G}_1&=\begin{pmatrix}1&0&0\\
0&1&-e^{M\left(\tilde{\theta}_3^N(z)-\tilde{\theta}_2^N(z)\right)}\\
0&0&1
\end{pmatrix},\quad
\mathcal{G}_2=\begin{pmatrix}1&0&0\\
0&1&0\\
0&-e^{M\left(\tilde{\theta}_2^N(z)-\tilde{\theta}_3^N(z)\right)}&1
\end{pmatrix},\\
\mathcal{G}_3&=\begin{pmatrix}1&0&0\\
-e^{M\left(\tilde{\theta}_1^N(z)-\tilde{\theta}_2^N(z)\right)}&1&0\\
0&0&1
\end{pmatrix},\quad
\mathcal{G}_4=\begin{pmatrix}1&0&0\\
0&1&0\\
-e^{M\left(\tilde{\theta}_1^N(z)-\tilde{\theta}_3^N(z)\right)}&0&1
\end{pmatrix}.
\end{split}
\end{equation}
We will now define the matrix $S(z)$ to be $S(z)=T(z)$ outside of
the lens region. Inside the lens region, we define $S(z)$ to be
\begin{equation}\label{eq:S1}
\begin{split}
\textrm{For $a<1$},\quad S(z)&=T(z)\mathcal{G}_2,\quad z\in
\mathbb{D}_1,\quad
S(z)=T(z)\mathcal{G}_1,\quad z\in \mathbb{D}_2,\\
S(z)&=T(z)\mathcal{G}_2\mathcal{G}_4^{\pm 1},\quad z\in
L^1_{\pm},\quad S(z)=T(z)\mathcal{G}_1\mathcal{G}_3^{\pm 1},\quad
z\in L^2_{\pm}.\\
\textrm{For $a>1$},\quad S(z)&=T(z)\mathcal{G}_1,\quad z\in
\mathbb{D}_1,\quad
S(z)=T(z)\mathcal{G}_2,\quad z\in \mathbb{D}_2,\\
S(z)&=T(z)\mathcal{G}_1\mathcal{G}_3^{\pm 1},\quad z\in
L^1_{\pm},\quad S(z)=T(z)\mathcal{G}_2\mathcal{G}_4^{\pm 1},\quad
z\in L^2_{\pm}.
\end{split}
\end{equation}
Let $\Xi$ be the union of the lens contours.
\begin{equation*}
\Xi=\Xi_L\cup\Xi_R\cup\left(\cup_{j=1}^2\Xi_j^{+}\cup\Xi_j^-\right).
\end{equation*}
Then by using Lemma \ref{le:cuttheta}, it is easy to check that
$S(z)$ satisfies the following Riemann-Hilbert problem.
\begin{equation}\label{eq:RHPS}
\begin{split}
1.\quad &\text{$S(z)$ is analytic in
$\mathbb{C}\setminus\left(\mathbb{R}_+\cup\Xi\cup\mathcal{C}\right)$},\\
2.\quad &S_+(z)=S_-(z)J_S(z),\quad z\in\left(\mathbb{R}_+\cup\Xi\cup\mathcal{C}\right),\\
3.\quad &S(z)=I+O(z^{-1}),\quad z\rightarrow\infty,\\
4.\quad  & S(z)=O(1),\quad z\rightarrow 0.
\end{split}
\end{equation}
where the matrix $J_S(z)$ is given by the following. Let
$\mathcal{C}=\cup_{j=0}^2\mathcal{C}_j$ where $\mathcal{C}_0$ is the
boundary between $\mathbb{D}_1$ and $\mathbb{D}_2$, $\mathcal{C}_1$
is the boundary between $L^1_-$ and $L^2_-$ and $\mathcal{C}_2$ is
the boundary between $L^1_+$ and $L^2_+$. Then on $\mathcal{C}$, the
jump matrix $J_S(z)$ is given by
\begin{equation}\label{eq:JSC}
\begin{split}
\textrm{For $a<1$},\quad J_S(z)&=\begin{pmatrix}1&0&0\\
0&0&1\\
0&-1&e^{M\left(\tilde{\theta}_{3,+}^N(z)-\tilde{\theta}_{3,-}^N(z)\right)}
\end{pmatrix},\quad z\in\mathcal{C}_0,\\
J_S(z)&=\begin{pmatrix}1&0&0\\
0&0&1\\
(-1)^{l-1}e^{M\left(\tilde{\theta}_{1}^N(z)-\tilde{\theta}_{3,-}^N(z)\right)}
&-1&e^{M\left(\tilde{\theta}_{3,+}^N(z)-\tilde{\theta}_{3,-}^N(z)\right)}
\end{pmatrix},\\ z&\in\mathcal{C}_l,\quad l=1,2.\\
\textrm{For $a>1$},\quad J_S(z)&=\begin{pmatrix}1&0&0\\
0&e^{M\left(\tilde{\theta}_{2,+}^N(z)-\tilde{\theta}_{2,-}^N(z)\right)}&-1\\
0&1&0
\end{pmatrix},\quad z\in\mathcal{C}_0,\\
J_S(z)&=\begin{pmatrix}1&0&0\\
(-1)^{l-1}e^{M\left(\tilde{\theta}_{1}^N(z)-\tilde{\theta}_{2,-}^N(z)\right)}
&e^{M\left(\tilde{\theta}_{2,+}^N(z)-\tilde{\theta}_{2,-}^N(z)\right)}&-1\\
0 &1&0
\end{pmatrix},\\ z&\in\mathcal{C}_l,\quad l=1,2.
\end{split}
\end{equation}
On the lens contour $\Xi$, the jump $J_S(z)$ is given by
\begin{equation}\label{eq:JSX}
\begin{split}
\textrm{For $a<1$},\quad J_S(z)&=\mathcal{G}_2^{-1},\quad
z\in\Xi_L,\quad J_S(z)=\mathcal{G}_1,\quad z\in\Xi_R,\\
J_S(z)&=\mathcal{G}_4^{-1},\quad z\in\Xi_{1}^{\pm},\quad
J_S(z)=\mathcal{G}_3^{-1},\quad z\in\Xi_{2}^{\pm},\\
\textrm{For $a>1$},\quad J_S(z)&=\mathcal{G}_1^{-1},\quad
z\in\Xi_L,\quad J_S(z)=\mathcal{G}_2,\quad z\in\Xi_R,\\
J_S(z)&=\mathcal{G}_3^{-1},\quad z\in\Xi_{1}^{\pm},\quad
J_S(z)=\mathcal{G}_4^{-1},\quad z\in\Xi_{2}^{\pm}.
\end{split}
\end{equation}
Finally, let $\mathbb{R}_{j}$ be the following.
\begin{equation}\label{eq:Rj}
\mathbb{R}_{1}=(0,\lambda_1^N),\quad\mathbb{R}_2=(\lambda_2^N,x_2).
\end{equation}
Then on the real axis, the jumps of $S(z)$ are given by
\begin{equation}\label{eq:JSreal}
\begin{split}
J_S(z)&=\begin{pmatrix}1&e^{M\left(\tilde{\theta}_{2,+}^N(z)-\tilde{\theta}_{1,-}^N(z)\right)}&0\\
0&1&0\\
0&0&1\end{pmatrix},\quad z\in\mathbb{R}_{k_2},\\
J_S(z)&=\begin{pmatrix}1&0&e^{M\left(\tilde{\theta}_{3,+}^N(z)-\tilde{\theta}_{1,-}^N(z)\right)}\\
0&1&0\\
0&0&1\end{pmatrix},\quad z\in\mathbb{R}_{k_3},\\
J_S(z)&=\begin{pmatrix}0&1&0\\
-1&0&0\\
0&0&1\end{pmatrix},\quad z\in\mathfrak{B}_{k_2},\quad
J_S(z)=\begin{pmatrix}0&0&1\\
0&1&0\\
-1&0&0\end{pmatrix},\quad z\in\mathfrak{B}_{k_3}.
\end{split}
\end{equation}
where $k_2=1$, $k_3=2$ for $a>1$ and $k_2=2$, $k_3=1$ for $a<1$. On
the rest of the positive real axis, the jump is given by
(\ref{eq:JTrest}).

By considering the behavior of $\theta_2^N(z)-\theta_3^N(z)$ near
$z=\infty$ in (\ref{eq:asymtheta}) and by making use of Lemma
\ref{le:diff23}, we see that for $a<1$ ($a>1$),
$\mathrm{Re}\left(\tilde{\theta}_2^N(z)-\tilde{\theta}_3^N(z)\right)$
is negative (positive) in $\Xi_L$ and positive (negative) in
$\Xi_R$. Hence, from (\ref{eq:JSX}), we see that the off-diagonal
entries in $J_S(z)$ become exponentially small on $\Xi_L$ and
$\Xi_R$ as $M\rightarrow\infty$ outside of some small neighborhoods
$D_{k}$ around the points $\lambda_{k}^N$. By Lemma \ref{le:lens},
we see that, away from the points $\lambda_k^N$, the jump matrices
on the lens contours $\Xi_1^{\pm}$ and $\Xi_2^{\pm}$ are all
exponentially close to the identity matrix as $M\rightarrow\infty$.
By (\ref{eq:size}) and the fact that $x_2>x_R$, it follows that
$J_S(z)$ are also exponentially close to the identity matrix on
$\mathbb{R}_+\setminus\left([\lambda_1^N,\lambda_2^N]\cup D_1\cup
D_2\right)$ as $M\rightarrow\infty$. Now by (\ref{eq:sizex}), we see
that the for $a<1$, the $31^{th}$ entry of $J_S(z)$ in
$\mathcal{C}_1\cup\mathcal{C}_2$ (\ref{eq:JSC}) is exponentially
small in $M$, while for $a>1$, the $21^{th}$ entry of $J_S(z)$ in
$\mathcal{C}_1\cup\mathcal{C}_2$ is exponentially small in $M$.
Finally, by (\ref{eq:sizeC}), we see that for $a<1$, the $33^{th}$
entry of $J_S(z)$ in $\mathcal{C}$ is also exponentially small in
$M$, while for $a>1$, the $22^{th}$ entry of $J_S(z)$ in
$\mathcal{C}$ is exponentially small in $M$. This suggests the
following approximation to the Riemann-Hilbert problem
(\ref{eq:RHPS}).
\begin{equation}\label{eq:RHPSinf}
\begin{split}
1.\quad &\text{$S^{\infty}(z)$ is analytic in
$\mathbb{C}\setminus\left([\lambda_{1}^N,\lambda_{2}^N]\cup\mathcal{C}\right)$},\\
2.\quad &S_+^{\infty}(z)=S_-^{\infty}(z)J_{\infty}(z),\quad z\in[\lambda_{1}^N,\lambda_{2}^N]\cup\mathcal{C},\\
3.\quad &S^{\infty}(z)=I+O(z^{-1}),\quad z\rightarrow\infty,\\
4.\quad
&S^{\infty}(z)=O\left((z-\lambda_j^N)^{-\frac{1}{4}}\right),\quad
z\rightarrow\lambda_j^N, \quad j=1,\ldots, 4.
\end{split}
\end{equation}
where the matrix $J_{\infty}(z)$ is the same as $J_S(z)$ on
$[\lambda_{1}^N,\lambda_{2}^N]$ and on $\mathcal{C}$, it is given by
\begin{equation}\label{eq:Jinfty}
\begin{split}
\textrm{For $a<1$},\quad J_{\infty}(z)&=\begin{pmatrix}1&0&0\\
0&0&1\\
0&-1&0
\end{pmatrix},\quad z\in\mathcal{C},\\
\textrm{For $a>1$},\quad J_{\infty}(z)&=\begin{pmatrix}1&0&0\\
0&0&-1\\
0&1&0
\end{pmatrix},\quad z\in\mathcal{C}.
\end{split}
\end{equation}
In the next section we will give an explicit solution to this
Riemann-Hilbert problem and we will eventually show that
$S^{\infty}(z)$ is a good approximation of $S(z)$ when $z$ is
outside of the small neighborhood $D_k$ of the branch point
$\lambda_k^N$.
\subsection{Outer parametrix}\label{se:outer}
Let $\Lie_N$ be the Riemann surface defined by (\ref{eq:curveN}) and
let $\Gamma_j$ be the images of $\mathfrak{B}_j$ on $\Lie_N$ under
the map $\xi_{1,+}^N(z)$. That is
\begin{equation}\label{eq:Gammaj}
\Gamma_j=\left\{(z,\xi)\in\Lie_N|\quad \xi=\xi_{1,+}^N(z),\quad
z\in\mathfrak{B}_j.\right\},\quad j=1,2.
\end{equation}
Similarly, we define $\Gamma_c$ to be the image of $\mathcal{C}$.
\begin{equation}\label{eq:Gammac}
\begin{split}
\Gamma_c=\left\{(z,\xi)\in\Lie_N|\quad \xi=\xi_{l_2,+}^N(z),\quad
z\in\mathcal{C}.\right\},
\end{split}
\end{equation}
where $l_1=2$, $l_2=3$ when $a>1$ and $l_1=3$, $l_2=2$ when $a<1$.

Let us now define the functions $S^{\infty}_k(\xi)$, $k=1,2,3$ to be
the following functions on $\Lie_N$.
\begin{equation}\label{eq:Sinfk}
\begin{split}
S^{\infty}_1(\xi)&=a\sqrt{\prod_{j=1}^4\gamma_j^N}\frac{(\xi+1)(\xi+a^{-1})}{\sqrt{\prod_{j=1}^4(\xi-\gamma_j^N)}},\\
S^{\infty}_2(\xi)&=\frac{a\sqrt{\prod_{j=1}^4(1+\gamma_j^N)}}{a-1}\frac{\xi(\xi+a^{-1})}{\sqrt{\prod_{j=1}^4(\xi-\gamma_j^N)}},\\
S^{\infty}_3(\xi)&=\frac{\sqrt{\prod_{j=1}^4(1+a\gamma_j^N)}}{1-a}\frac{\xi(\xi+1)}{\sqrt{\prod_{j=1}^4(\xi-\gamma_j^N)}}.
\end{split}
\end{equation}
where $\gamma_k^N$ are the roots of polynomial (\ref{eq:quartN}).
The branch cuts of the square root in (\ref{eq:Sinfk}) are chosen to
be the contours $\Gamma_1$, $\Gamma_2$ (\ref{eq:Gammaj}) and
$\Gamma_c$ (\ref{eq:Gammac}) in $\Lie^N$.

By using the asymptotic behavior of the functions $\xi_m^N(z)$
(\ref{eq:xiinfty}) and (\ref{eq:zeroasym}), with $c$ and $\beta$
replaced by $c_N$ and $\beta_N$, we see that the all the functions
$S^{\infty}_k(\xi)$ are holomorphic near $\xi_m^N(0)$ for $m=1,2,3$.
Moreover, at the points $\xi_m^N(\infty)$, these functions satisfy
\begin{equation}\label{eq:sasym}
\begin{split}
S^{\infty}_k(\xi_m^N(\infty))=\delta_{mk},\quad k,m=1,2,3.
\end{split}
\end{equation}
Let us define $S^{\infty}(z)$ to be the following matrix on
$z\in\mathbb{C}$.
\begin{equation}\label{eq:sinf}
\left(S^{\infty}(z)\right)_{im}=S^{\infty}_i(\xi_m^N(z)),\quad 1\leq
i,m\leq 3.
\end{equation}
Then, since the branch cut of the square root in (\ref{eq:Sinfk})
are chosen to be $\Gamma_j$ and $\Gamma_c$, we have, from the jump
discontinuities of the $\xi_m^N(z)$ (\ref{eq:bound}), the following
\begin{equation}\label{eq:Sjump}
\begin{split}
S^{\infty}_i(\xi_{1,\pm}^N(z))&=\mp
S^{\infty}_i(\xi_{l,\mp}^N(z)),\quad
z\in\mathfrak{B}_{k_l},\quad i=1,2,3,\quad l=2,3,\\
S^{\infty}_i(\xi_{l_2,{\pm}}^N(z))&=\mp
S^{\infty}_i(\xi_{l_1,\mp}^N(z)),\quad z\in\mathcal{C},\quad
i=1,2,3.
\end{split}
\end{equation}
From this and the asymptotic behavior (\ref{eq:sasym}) of the
$S^{\infty}_i(\xi)$ at $z=\infty$ and its behavior at the branch
points $\lambda_j^N$, we see that the matrix $S^{\infty}(z)$
satisfies the Riemann-Hilbert problem (\ref{eq:RHPS}).
\begin{proposition}\label{pro:outer}
The matrix $S^{\infty}(z)$ defined by (\ref{eq:sinf}) satisfies the
Riemann-Hilbert problem (\ref{eq:RHPSinf}).
\end{proposition}
\subsection{Local parametrices near the edge points
$\lambda_k^N$}\label{se:local}

Near the edge points $\lambda_k^N$, the approximation of $S(z)$ by
$S^{\infty}(z)$ failed and we must solve the Riemann-Hilbert problem
exactly near these points and match the solutions to the outer
parametrix (\ref{eq:sinf}) up to an error term of order $O(M^{-1})$.
To be precise, let $\delta>0$ and let $D_k$ be a disc of radius
$\delta$ centered at the point $\lambda_k^N$, $k=1,\ldots,4$. We
would like to construct local parametrices $S^k(z)$ in $D_k$ such
that
\begin{equation}\label{eq:localpara}
\begin{split}
1.\quad &\text{$S^{k}(z)$ is analytic in
$D_k\setminus \left(D_k\cap\left(\mathbb{R}\cup\Xi\cup\mathcal{C}\right)\right)$},\\
2.\quad &S_+^{k}(z)=S_-^{k}(z)J_S(z),\quad z\in D_k\cap\left(\mathbb{R}\cup\Xi\cup\mathcal{C}\right),\\
3.\quad &S^{k}(z)=\left(I+O(M^{-1})\right)S^{\infty}(z),\quad z\in\p
D_k.
\end{split}
\end{equation}
The local parametrices $S^k(z)$ can be constructed by using the Airy
function as in \cite{BKext2} (See also \cite{Lysov}). Since the
construction is identical to that in \cite{BKext2} and \cite{Lysov},
we shall not repeat the details here and refer the readers to these
2 papers.
\subsection{Last transformation of the Riemann-Hilbert problem}
Let us now show that the parametrices constructed in Section
\ref{se:outer} and Section \ref{se:local} are indeed good
approximation to the solution $S(z)$ of the Riemann-Hilbert problem
(\ref{eq:RHPS}).

Let us define $R(z)$ to be the following matrix
\begin{equation}\label{eq:Rx}
\begin{split}
R(z)=\left\{
       \begin{array}{ll}
         S(z)\left(S^{k}(z)\right)^{-1}, & \hbox{$z$ inside $D_k$, $k=1,\ldots,4$;} \\
         S(z)\left(S^{\infty}(z)\right)^{-1}, & \hbox{$z$ outside of $D_k$, $k=1,\ldots,4$.}
       \end{array}
     \right.
\end{split}
\end{equation}
Then the function $R(z)$ has jump discontinuities on the contour
$\Gamma_R$ shown in Figure \ref{fig:gamma}.
\begin{figure}
\centering
\includegraphics[scale=0.5]{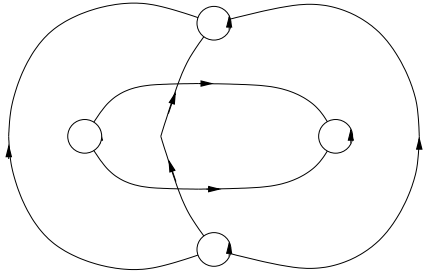}
\caption{The contour $\Gamma_R$.}\label{fig:gamma}
\end{figure}
In particular, $R(z)$ satisfies the Riemann-Hilbert problem
\begin{equation}\label{eq:RHR}
\begin{split}
&1. \quad \text{$R(z)$ is analytic in $\mathbb{C}\setminus\Gamma_R$}\\
&2.\quad R_+(z)=R_-(z)J_R(z)\\
&3. \quad R(z)=I+O(z^{-1}),\quad z\rightarrow\infty, \\
&4.\quad R(z)=O(1),\quad z\rightarrow 0.
\end{split}
\end{equation}
From the definition of $R(z)$ (\ref{eq:Rx}), it is easy to see that
the jumps $J_R(z)$ has the following order of magnitude.
\begin{equation}\label{eq:Jx}
\begin{split}
J_R(z)=\left\{
       \begin{array}{ll}
         I+O(M^{-1}), & \hbox{$z\in\p D_k$, $k=1,\ldots, 4$ ;} \\
         I+O\left(e^{-M\eta}\right), & \hbox{for some fixed $\eta>0$ on
the rest of $\Gamma_R$.}
       \end{array}
     \right.
\end{split}
\end{equation}
Then by the standard theory, \cite{D}, \cite{DKV}, \cite{DKV2}, we
have
\begin{equation}\label{eq:Rest}
\begin{split}
R(z)=I+O\left(\frac{1}{M(|z|+1)}\right),
\end{split}
\end{equation}
uniformly in $\mathbb{C}$.

In particular, the solution $S(z)$ of the Riemann-Hilbert problem
(\ref{eq:RHPS}) can be approximated by $S^{\infty}(z)$ and
$S^{k}(z)$ as
\begin{equation}\label{eq:approxS}
\begin{split}
S(z)=\left\{
       \begin{array}{ll}
         \left(I+O\left(M^{-1}\right)\right)S^{k}(z), & \hbox{$z\in D_k$, $k=1,\ldots,4$;} \\
         \left(I+O\left(M^{-1}\right)\right)S^{\infty}(z), & \hbox{$z$ outside of
$D_k$, $k=1,\ldots,4$.}
       \end{array}
     \right.
\end{split}
\end{equation}
\section{Universality of the correlation kernel}
The the universality result Theorem \ref{thm:main2} can now be
proven by using the asymptotics of the multiple Laguerre polynomials
obtained in the last section. The proof is the same as the ones in
Section 9 of \cite{BKext1}, \cite{BKext2} and \cite{Mo} and we shall
leave the readers to verify the details.
\section{Conclusions}
This paper complements the results obtained in our earlier paper
\cite{Mo} on the universality of complex Wishart ensembles 2 with
distinct eigenvalues and the number of each eigenvalue becomes large
as the size of the ensemble goes to infinity. By using results in
the Stieltjes transform of the limiting eigenvalue distribution, we
were able to overcome difficulties in determining the sheet
structure of the Riemann surface (\ref{eq:curve20}). This, together
with an analysis of the topology of the zero set of a function
$h(x)$ (\ref{eq:hx}), enables us to apply the Deift-Zhou steepest
descent method to obtain asymptotic formula for the eigenvalue
correlation function. Together with \cite{Mo}, we showed that unless
the discriminant of the polynomial (\ref{eq:quartic1}) is zero, the
local eigenvalue statistics are given by the sine-kernel
(\ref{eq:bulk}) in the bulk and the Airy kernel (\ref{eq:edge}) in
the edge of the spectrum. We have also shown that the largest
eigenvalue is distributed according to the Tracy-Widom distribution
(\ref{eq:TW}).

When the discriminant $\Delta$ of (\ref{eq:quartic1}) becomes zero,
the ensemble goes through a phase transition in which the support of
the limiting eigenvalue density changes from 1 interval to 2
intervals. This phase transition can be studied by the method
developed in \cite{BKdou} and the eigenvalue correlation function
will be given by the Pearcey kernel near the critical point of the
spectrum, the point where the support is splitting.

The use of Stieltjes transform to provide a Riemann surface needed
for the implementation of the Riemann-Hilbert analysis can be
generalized to cases where the covariance matrix has more than 2
distinct eigenvalues. However, when the Riemann surface has more
than 2 complex branch points, the determination of the zero set
$\frak{H}$ in (\ref{eq:frakH}) is very complicated and a full
generalization to these cases is still a very challenging open
problem.

\vspace{.25cm}

\noindent\rule{16.2cm}{.5pt}

\vspace{.25cm}

{\small

\noindent {\sl School of Mathematics \\
                       University of Bristol\\
                       Bristol BS8 1TW, UK  \\
                       Email: {\tt m.mo@bristol.ac.uk}

                       \vspace{.25cm}

                       \noindent  22 September  2008}}

\end{document}